\documentclass[contentspage]{article}
\usepackage{amsmath,amssymb,amsthm,varioref,enumerate}
\usepackage{bm}
\usepackage{tikz}
\usepackage[backref=page]{hyperref}
\newtheorem{theorem}{Theorem}

\newtheorem{lemma}{Lemma}
\newtheorem{definition}{Definition}

\def\seg{{\rm Seg}}

\def\toggle{{\rm Toggle}}
\def\vs{v^\#}
\def\fs{{f^*}}
\def\fss{{f^{**}}}

\newcommand{\comment}[1]{}

\def\be{\bm{e}}
\def\ee{{e'}}
\def\gh{\hat{g}}

\def\region{{\rm Region}}
\def\choice{{\rm Choice}}
\def\nil{{\rm null}}
\def\candidates{{\rm Candidates}}
\def\vtwiddle{{V^\#}}     
\def\toggle{{\rm Toggle}}
\def\cA{\mathcal{A}}
\def\distance{{\rm Distance}}
\def\dtv{d_{\rm TV}}
\def\Gkt{G_{(k,t)}}  
\def\Gt{G_{(t)}}

\def\sedit{\mbox{{\footnotesize Q-EDIT}}}

\def\ktsedit{{\sedit_{k,t}}}

\def\rank{\mbox{{\footnotesize  RANK}}}  
\def\cS{\mathcal{S}}
\def\bL{\bm{L}}   
\def\bM{\bm{M}}

\def\bbM{\overline{\bM}}
\def\barM{\overline{M}}
\def\bbL{\overline{\bL}}
\def\barL{\overline{L}}
\def\bA{\bm{A}}
\def\bbA{\overline{\bA}}

\def\bbA{\overline{\bA}}
\def\bC{\bm{C}}
\def\bbC{\overline{\bC}}

\def\bD{\bm{D}}
\def\bbD{\overline{\bD}}
\def\barD{\overline{D}}

\def\bX{\bm{X}}
\def\bXv{\bX^{(v)}}
\def\bXp{\bX'}
\def\bx{\bm{x}}
\def\by{\bm{y}}

\def\bp{\bm{p}}
\def\var{{\rm Var \,}}

\def\e{\mathbb{E}\,}
\def\p{\mathbb{P}}

\def\P{\mathbb{P}}

\title{Random Feedback Shift Registers, and the Limit Distribution for
Largest Cycle Lengths}

\author{Richard A. Arratia, University of Southern California \and
  E. Rodney Canfield, University of Georgia \and
  Alfred W. Hales, Center for Communications Research, La Jolla}

\begin{document}
\date{July 22, 2022}
\maketitle

\begin{abstract}
For a random binary noncoalescing feedback shift register of
width $n$, with all $2^{2^{n-1}}$ possible feedback functions $f$
equally likely, the process of long cycle lengths, scaled by dividing
by $N=2^n$, converges in distribution to the same
Poisson-Dirichlet limit as holds for random permutations in
$\mathcal{S}_N$, with all $N!$ possible permutations equally likely.
  Such behavior was conjectured by Golomb, Welch, and Goldstein in
  1959.

\end{abstract}
\tableofcontents

\clearpage

\section{Introduction}\label{sect intro}

We consider feedback shift registers, linear in the eldest bit 
(in $\mathbb{F}_2$), given as
\begin{equation}\label{feedback}
   x_{t+n} = x_t \oplus f(x_{t+1},x_{t+2},\ldots,x_{t+n-1}).
\end{equation}
Here

\begin{equation}\label{def f}
   f:  \mathbb{F}_2^{n-1} \to \mathbb{F}_2
\end{equation}
is an arbitrary $n-1$ bit Boolean function (the ``feedback'' or ``logic''), 
and we will consider all
$2^{2^{n-1}}$ possible $f$ to be equally likely.  We write
$$
   N := 2^n
$$
and note that the map
\begin{eqnarray}
\pi_f:   \mathbb{F}_2^n  & \to &\mathbb{F}_2^n  \label{def pi f}
\nonumber \\
     (x_0,x_1,\ldots,x_{n-1}) & \mapsto &  (x_1,\ldots,x_{n-1},x_n) \\
                             & = &
(x_1,\ldots,x_{n-1},x_0 \oplus f(x_1,\ldots,x_{n-1}))  \nonumber
\end{eqnarray}
is a permutation on $N$ objects.

In 1959 \cite{GolombWelch}, see also Chapter VII of \cite{Golomb},
Golomb, Welch, and Goldstein suggest 
that
the flat random permutation in $\cS_N$, with all $N!$ permutations $\pi$
equally likely, gives a good approximation to the cycle structure of $\pi_f$, in the sense that the
cycle structure of $\pi_f$ is close to the cycle structure of $\pi$, 
in various aspects of distribution, such as the average length of the
longest cycle.    See \cite{wolfram}, especially the section ``Cellular Automata and Nonlinear Shift Registers,'' which includes an anecdote that Golomb used custom hardware modules in 1956 to experiment on this conjecture, and these ran  about 3 million times faster than the general purpose computer on the same problem.

We prove that the longest cycle part of this conjecture  
is true, and more, namely  that $\pi$ and $\pi_f$
have the same limit distributions in the infinite-dimensional simplex
$\Delta$,
for the processes\footnote{A (stochastic) process is simply a collection of
  random variables, or, depending on one's point of view, the joint
  distribution of that collection.\label{process footnote}} of
long  
cycle lengths, scaled by $N$. This does not answer other
aspects of Golomb's conjecture, involving the distribution of the number
of cycles, or behavior of short cycles.

 There are two natural ways to view the large cycles of the random
 permutation $\pi_f$, which we now describe briefly.  First, there is
 the process of largest cycle lengths: write $L_i$ for the length of
 the $i^\mathrm{th}$ longest cycle of $\pi_f$, with $L_i := 0$ if the
 permutation has fewer than $i$ cycles, so that always $L_1+L_2+\cdots
 = N$, where $N=2^n$.  Write $\bbL=\bbL(N)$ for the process of scaled
 cycle lengths, $\bbL=(L_1/N,L_2/N,\ldots)$.  Second, there is the
 process of cycle lengths taken in \emph{age order}: pick a random
 $n$-tuple, take $A_1$ to be the length of the cycle of $\pi_f$
 containing that first $n$-tuple, then pick a random $n$-tuple from
 among those not on the first cycle, take $A_2$ to be the length of
 the cycle of $\pi_f$ containing that second $n$-tuple, and so on.
 Write $\bbA=\bbA(N)=(A_1/N,A_2/N,\ldots)$ for the process of scaled
 cycle lengths in \emph{age order}.  For flat random permutations
 $\pi$ in place of $\pi_f$, the limit of $\bbA$ is called the GEM
 process (after Griffiths~\cite{griffiths}, Engen~\cite{engen}, and
 McCloskey~\cite{mccloskey}); it is the distribution of
 $(1-U_1,U_1(1-U_2),U_1U_2(1-U_3),\ldots)$, where 
 $U,U_1,U_2,\ldots$ are independent and uniformly distributed in
 $(0,1)$.  The Poisson-Dirichlet process is 
 $(X_1,X_2,\ldots)$ where $X_i$ is the $i^{\rm th}$ largest of
 $1-U_1,U_1(1-U_2),U_1U_2(1-U_3),\ldots$.  This construction gives the
 simplest way to characterize the Poisson-Dirichlet process, PD.  For
 flat random permutations, the limit of $\bbL$ is
PD.\footnote{This same Poisson-Dirichlet
process  also gives the distributional limit for the process of 
scaled bit sizes of the prime factors 
of an integer
chosen uniformly  from 1 to $x$,  as $x$ goes to infinity.
Here we write PD for PD$(1)$, where,
in general, GEM$(\theta)$ and PD$(\theta)$ for $\theta>0$ are
constructed using $U^{1/\theta}$ in place of $U$, and the case
$\theta=1/2$ gives the limits for the processes of sizes of largest
components, in age order or strict size order, for random mappings,
\textit{i.e.},
functions from $[n]$ to $[n]$ with all $n^n$ possibilities equally likely.
}
See Section~\ref{sect background} for a review of these
concepts, including more discussion of age-order and the GEM limit as used
in~\eqref{thm 1 GEM}.  See also \cite{ABTbook}.  
Formally, our result is the following:
\begin{theorem}\label{theorem 1}
Consider the random permutation $\pi_f$ given by~\eqref{def pi f}, where all 
$2^{2^{n-1}}$ possible $f$ in~\eqref{def f} are equally likely.  
Then, as $n \to \infty$, $\bbL(N)$ converges in distribution to $(X_1,X_2,\ldots)$ with PD distribution.
\end{theorem}

Writing $\to^d$ to denote convergence in distribution, we can succinctly summarize the conclusion of Theorem \ref{theorem 1} by writing
\begin{equation}
    \bbL(N) \to^d {\bf X} :=(X_1,X_2,\ldots).
\label{thm 1 conclusion}
\end{equation}

We note some easy consequences of Theorem~\ref{theorem 1}.
Theorem~\ref{theorem 1} is equivalent to 
\begin{equation}
    \bbA(N) \to^d  (1-U_1,U_1(1-U_2),U_1U_2(1-U_3),\ldots),
\label{thm 1 GEM}
\end{equation}
with GEM distribution,
by a soft
argument involving size-biased permutations, originally given by 
\cite{Donnelly}.
By projecting onto the first coordinate\footnote{Since $U, U_1$ and $1-U_1$ all have the same
  distribution, uniform in $(0,1)$.}, we see
\begin{equation}
    \frac{A_1}{N} \to^d  U.
\label{thm 1 A}
\end{equation}
By taking expectations, we see 
\begin{equation}
   \e \ \frac{A_1}{N} \to \frac{1}{2}.
\label{thm 1 E}
\end{equation}

Of course, the uniform distributional limit in 
\eqref{thm 1 A} makes no \emph{local} limit claim; it is plausible that
$N \, \p(A_1 = i) \to 1$ holds uniformly in $n < i < N-n$.  For any
fixed $i > 1$, the statement $N \, \p(A_1 = i) \to 1$ is
\emph{false}.  It is true that  
$N \, \p(A_1 = 1) = N \, \p(A_1 = N) =1$.  And for any
fixed $j>0$ the statement $N \, \p(A_1 = N-j) \to 1$
is \emph{false}; see \cite{punctured}.



   We work
with the de Bruijn graph $D_{n-1}$, 
with edge set  $\mathbb{F}_2^n$ and vertex set  $\mathbb{F}_2^{n-1}$;
edge $e = (y_0,y_1,\ldots,y_{n-1})$ goes \emph{from}  vertex 
$v=  (y_0,y_1,\ldots,y_{n-2})$ \emph{to}  vertex 
$v'=  (y_1,\ldots,y_{n-1})$.
The graph $D_{n-1}$ is 2-in, 2-out regular, and a 
random feedback logic $f$ corresponds to a \emph{random
  resolution} of all vertices;  the resolution at a vertex $v$ pairs
the incoming edges, $0v$ and $1v$, with the outgoing edges $v0$ and $v1$.  The
cycles of a random permutation $\pi_f$ correspond exactly to 
the edge-disjoint cycles in a random circuit
decomposition of the Eulerian graph $D_{n-1}$.

\section{A Survey of the Proof of Theorem~\ref{theorem 1}}
\label{sect survey}
\def\D{{\cal E}}
\def\eps{\epsilon}
\def\V{{\mathbb F_2}}
\def\S{\mathcal{S}}
\def\Snk{S_{n,k}}
\def\rel{{\rm rel}}

  In this section we survey the proof of Theorem~\ref{theorem 1}
while omitting many necessary technicalities. It is hoped the reader
will thus have
a better notion of what is happening, and why,
as s/he reads the later sections.
We begin with the notion of
{\it relativisation}.  Suppose, as for example
in the hypotheses of Theorem~\ref{theorem 1},
that one has for each $n=1,2,\dots$ a probability $P_n$
on the permutations 
of a set $\D_n$.  Let $\pi\in\S(\D_n)$ be one such
permutation, and let $\be=(e_1,\dots,e_k)$ be a $k$-tuple of,
for now, distinct
elements from the domain $\D_n$.  Picturing the permutation
$\pi$
as a collection of disjoint cycles, one sees that by ignoring
all
elements of $\D_n$ except for the $e_i$, these latter are permuted
among themselves.  That is, starting with $e_i$, traverse the cycle
of $\pi$ containing this element: -- 
$$
e_i,\pi(e_i),\pi^2(e_i),\dots
$$
until after one or more steps an element $e_j$ is encountered.
(It is possible for the first element so encountered
to be $e_i$, which happens when the traversed cycle contains
only a single member of the $k$-tuple $\be$.)  Since the
$e_i$ are
given in a definite order, the induced permutation among these elements
is readily identified with an element of $\S_k$, the permutations of
the set $\{1,2,\dots,k\}$.   Altogether, we have a function
$$
\rel_{n,k} : \S(D_n) \times (D_n)_k ~\rightarrow~
\S_k
$$
which we call {\it relativisation}.  Here, $(D_n)_k$ denotes the
ordered $k$-tuples drawn from $\D_n$ without replacement.
We shall prove:
Suppose that for every
fixed $k\ge 1$ the
sequence of distributions induced on $\S_k$ by the
functions $\rel_{n,k}$ and the probability distributions
$P_n$ tends to the uniform distribution.
(For brevity, we say
 ``$P_n$ has the uniform relativisation property.'')
Then the large cycle
process associated with $P_n$ tends to Poisson-Dirichlet.
The proof that the uniform relativisation property implies
the Poisson-Dirichlet property appears in Section~\ref{sect perm version},
as Lemma~\ref{permutation lemma}.

 Henceforth
we specialize to the particular sequence $P_n$ of interest:
the sets $\D_n$ are the binary $n$-tuples
$\V^n$, and $P_n$ assigns equal weight to each of the $2^{N/2}$
shift register
permutations $\pi_f$ (and no weight to other permutations), where
$N=2^n$.
Let $\Snk$ denote the Cartesian product
$$
\{ f:\V^{n-1}\rightarrow\V\} ~\times~ (\V^n)^k \,.
$$
For technical reasons we define the relativisation
function $\rel_{n,k}$ on the set
$\Snk$, see Definition~\eqref{relativisation}
in Section~\ref{sect tangle}.  Nevertheless,
pairs $(f,\e)$ in which $\be$ contains a repeated element may be
safely ignored by the reader for now, and only
the primary objective be kept in mind:
to show that as $(f,\be)$ varies over $\Snk$
the coverage of $\S_k$ under the relativisation function
$\rel_{n,k}$ is
approximately uniform.  

 Roughly speaking, this objective is accomplished by
partitioning the set $\Snk$ into blocks such that the restriction
of $\rel_{n,k}$ to each block of the partition yields an almost
uniform coverage of $\S_k$.  The description of these blocks involves
the notion of {\it toggle}.
Let
$v\in\V^{n-1}$ and $f$ be a feedback function; then the {\it toggle of
the function $f$ at the point $v$} is the function $f_v$ which
disagrees with $f$ only at the argument $v$:
$$
f_v(w) = \begin{cases}  f(w)  &  w\neq v \\
                 1\oplus f(w) &  w=v  \, .
         \end{cases}
$$
That is, we have toggled a single bit in the truth table of $f$.
Toggling a feedback function has a predictable effect on
$\rel_{n,k}(\pi_f,\be)$.  In particular, for $x\in\V,
v\in\V^{n-1}$, if $xv=\pi_f^i(e_a)$ and
$$
\{\pi_f(e_a),\pi_f^2(e_a),\dots,\pi_f^i(e_a)\}
~\cap~ \{e_1,\dots,e_k\} ~~ = ~~ \emptyset,
$$
and, (with $a<b$), $\bar{x}v=\pi_f^j(e_b)$, and
$$
\{\pi_f(e_b),\pi_f^2(e_b),\dots,\pi_f^j(e_b)\}
~\cap~ \{e_1,\dots,e_k\} ~~ = ~~ \emptyset,
$$
then (let the reader check by drawing a picture)
$$
\rel_{n,k}(\pi_{f_v},\be) ~~=~~ \rel_{n,k}(\pi_{f},\be)\circ(a,b),
$$
where $(a,b)$ denotes a transposition in $\S_k$.  The blocks in
our partition of $\Snk$ arise as follows: given $(f,\be)\in\Snk$,
we determine, in a way explained below, a subset of size $m$,
$$
V=\{v_1^{\#},\dots,v_m^{\#}\}\subseteq\V^{n-1} \, ,
$$
and define the block containing $(f,\be)$ to be the $2^m$ different
toggles $(f_U,\be)$, $U$ ranging over subsets of $V$.  Here,
$f_U$ denotes function $f$ toggled at all $v\in U$.  For the
block to be well defined, it must be the case that the choice of
$V$ will be the same for all $f_U$ as for $f$.  This necessitates the
introduction of a subset $H\subseteq\Snk$, the ``happy event,''
see equation \eqref{def H} in Section~\ref{sect happy event}.
It turns out that
the happy event is almost all of $\Snk$, $|H|/|\Snk|\to 1$, and
for $(f,\be)\in H$ the blocks are well defined. Moreover for each
such 
block we have an ordered sequence of
transpositions $(a_i, b_i)\in\S_k$ ($1\le i\le m)$
with
$$
\rel_{n,k}(\pi_{f_U},\be) = \rel_{n,k}(f,\be)
\circ (a_{i_1},b_{i_1})\circ\cdots\circ(a_{i_{\ell}},b_{i_{\ell}})
$$
where $U=\{i_1,\dots,i_{\ell}\}$.  For $m$ sufficiently large,
almost all such sequences of transpositions yield $2^m$
compositions which cover $\S_k$
almost uniformly.  (Lemma~\ref{schedule lemma} in
Section~\ref{sect schedule} proves
that for all $k,\eps$
there is an $m$ such that the distribution induced on $\S_k$
is within $\eps$ of uniform in total variation for all but an
$\eps$ fraction of possible sequences.)

  Let us say something about how, given $(f,\be)\in\Snk$,
the $m$-subset $V$ of $\V^{n-1}$ is chosen.  The pair $(f,\be)$
determines $k$ segments
\begin{equation}
\label{eq segments}
e_a, \pi_f(e_a), \dots, \pi_f^t(e_a)   ~~~ (1\le a\le k)
\end{equation}
in which the length
$t$ is taken to be
approximately $N^{3/5}$.   For this length it is almost certain
that not only are the initial edges
$e_a$ distinct, but in fact all $k\times(t+1)$ of the
edges $\pi_f^i(e_a)$ are distinct.  This feature is
included in the definition~\eqref{def H} of
event $H$.  Given that $\pi_f$ acts by shifting left and bringing in
one new bit on the right,
each sequence \eqref{eq segments} is equivalent to a binary sequence
$$
e_{a,0} e_{a,1} \cdots e_{a,n-1} \cdots e_{a,n+t-1}
$$
of length $n+t$.  To be considered for membership in $V$, an 
$(n-1)$-tuple
$v^{\#}$ must appear in two of these binary sequences; that is,
for some $a<b$ and some bit $x\in\{0,1\}$
$$
xv^{\#} = e_{a,i}\cdots e_{a,i+n-1}
~{\rm and}~~
\bar{x}v^{\#} = e_{b,j}\cdots e_{b,j+n-1}.
$$

 One may ask as $(f,\be)$ varies uniformly over $\Snk$ what is the
probability of finding such leftmost $(n-1)$-repeats $(i,j)$ in
various regions of the plane?    Remarkably, such points when
rescaled as $(i/N^{1/2},j/N^{1/2})$ constitute, in the limit with
respect to total variation distance, a familiar Poisson process.
Thanks to this limiting behavior
we can estimate not only the probability of finding $v^{\#}$'s
which satisfy the above minimal constraint for $V$-membership,
but also the probability of finding $m$ $v^{\#}$'s lying
in a much more stringently constrained geometry, which geometry
implies $(f,\be)\in H$.  Section~\ref{sect toggling}
is devoted to proving these
properties of $H$ and $V$ under the assumption that the probabilities
in question can be approximated by a Poisson process.

 We conclude our survey by saying how this last assumption
is justified.  We present in
Section~\ref{sect coins} an algorithm called
{\it sequential editing} which begins with $k$ random binary
sequences (referred to as coin-toss sequences) and edits
them in such a way that the result of the editing is a set of
$k$ sequences which {\it could have} been produced by choosing
$(f,\be)\in\Snk$ and forming the $k$ segments \eqref{eq segments}.
Even more,
the probability of obtaining a particular set of sequences is
exactly the same, whether we choose $(f,\be)$ and form
\eqref{eq segments}, or
flip $k(n+t)$ coins and perform sequential editing.  (This
is proven in Theorem~\ref{theorem honest t} of
Section~\ref{sect induced probability}).

 Moreover, there is a ``good event'' $G$,
$G\subseteq\V^{(n+t)k}$, such that when the initial coin toss
sequence $C$
belongs to $G$ the leftmost $(n-1)$-repeats in the edited
sequence appear in exactly the same locations $(i,j)$ as they
do in $C$.   Since $G$ is almost all of $\V^{(n+t)k}$
(Theorem~\ref{probG} in Section~\ref{sect coins}),
the study of the $(f,\be)$-induced
pairs is reduced to the study of leftmost $(n-1)$-repeats in
$k$ random sequences.  This new process is by no means easily
evaluated, but fortunately it is in the realm of the Chen-Stein
method as presented and extended in \cite{AGG89}.  In such a manner
the above described approximation is justified.

 Looking back at this survey, it appears that
the components
in the proof of
Theorem 1 have been described almost in the reverse order 
that they appear in the
sequel.  May we wish that in the end the determined reader will
understand the proof forwards and backwards.

\section{Comparisons with Coin Tossing Sequences}
\label{sect coins}

Throughout this section these conventions will be observed: $a_i,b_i,C_i$ denote
bits; $v_i$ denotes an $(n-1)$-long sequence of bits; and $e_i$
denotes an $n$-long sequence of bits.
A tool used in the proof of Theorem~\ref{theorem 1} is to
compare the bit sequence $b_0,b_1,\dots b_{n+t}$ generated by a randomly chosen
feedback logic $f$ with a coin toss sequence, denoted in this section
$C_0,C_1,\dots, C_{n+t}$.  A bit sequence $b_i$ generated by a feedback logic has what we
refer to as the {\it de Bruijn property}:
it satisfies a recursion of the form
$b_{t+n} = b_t + f(b_{t+1},\dots,b_{t+n-1})$.  In a sequence with the de Bruijn property the
$n$-long words $0v$ and $1v$ must be followed by different bits.   Of course, not every coin
toss sequence has the de Bruijn property.   The {\it sequential edit}, defined 
below, of a coin toss sequence $C_i$ is
obtained in a left-to-right bit-by-bit manner and adheres as closely as possible to $C_i$, changes
being made only when forced by the desire to respect the de Bruijn property.
On the other hand, 
the {\it shotgun edit}, also defined below, of a sequence $C_i$ is a naive
imitation of a sequential edit.  
In a sense and circumstances to be made precise, 
by the combination of Theorems \ref{detG} and \ref{probG},
with high probability, these two produce the same output.

\subsection{Sequential Editing}\label{sect sequential}

\noindent We begin with an $n+t$ long bit sequence
$$
C_{0}, C_1, \dots, C_{t+n-1}.
$$
The new bit sequence of the same length,
$$
b_{0}, b_1, \dots, b_{t+n-1}
$$
is produced by following two rules:
\begin{description}
\item[Rule 1:]
$$ 
b_i = C_i, ~ 0 \le i \le n-1;
$$
\item[Rule 2:] For $i\ge 0$ determine bit $b_{i+n}$ 
by first asking if the feedback logic bit $f(b_{i+1},\dots,b_{i+n-1})$ has
been previously defined; if so, set $b_{i+n}$ accordingly:
$$
b_{i+n} = b_{i} \oplus f(b_{i+1},\dots,b_{i+n-1});
$$
otherwise, define (and remember) the feedback logic bit 
in such a way that $b_{i+n}$ and $C_{i+n}$ agree.
\end{description}

\noindent   Here we give some terminology, and indexing practice.  
We say that the sequence $b$ is obtained from
the coin toss sequence $C$ by {\it sequential editing}.  Each time a $b_{i+n}$ has
freedom -- because the necessary feedback bit has not yet been set -- we set the
feedback bit so that $b_{i+n}=C_{i+n}$; but at any time the bit $b_{i+n}$ ``has no choice,'' we
assign it the forced value.  Such a time $i$ is a time of a {\it potential edit}; if
it turns out (by chance) that $b_{i+n}$ and $C_{i+n}$
agree, then no {\it actual edit} has taken place; if it is forced to take $b_{i+n}$ equal to
$\overline{C_{i+n}}$ then an actual edit has taken place, and we
label the time of this actual edit as $i$ rather than $i+n$.   The sequence
$b$ obtained by this process always has the de Bruijn property.  In
terms of the de Bruijn graph with all vertices resolved, a potential
edit occurs at time $i$  when $e_{i}$, the edge from $v_{i}$ to
$v_{i+1}$, is going in to a vertex $v=v_{i+1}$ where $f(v)$, the
resolution of that vertex, is already known,
so that the successor edge, $e_{i+1}=\pi_f(e_i)$ is determined ---
this is equivalent to determining $b_{i+n}$, the rightmost bit of $e_{i+1}$.

\subsection{Shotgun Editing}

\noindent Now we define a second, generally different, way to edit the coin toss sequence
$C_i$ to produce a sequence $a_i$.  We call this the {\it shotgun edit}.
Unlike $b_i$ obtained by sequential editing, the sequence $a_i$ obtained by shotgun
editing may not have the de Bruijn property.   

\vskip 10pt

\noindent The symbols $I$, $J$, $I_k$, $J_k$ denote intervals of
integers contained in the $(n+t)$-long interval $[0,1,2,\dots,t+n-1]$.  We use
$\ell(I)$ and $r(I)$ to denote the left- and right- endpoints of the interval
$I$.   A binary sequence
\begin{equation}\label{eq 1}
C_{0},C_1, \dots, C_{t+n-1}
\end{equation}
has an $m$-{\it long repeat} at $(I,J)$ if $\ell(I)<\ell(J)$, $|I|=m=|J|$ and the two ordered
$m$-tuples $(C_i:i\in I)$ and
$(C_j:j\in J)$ are equal.  We say that (\ref{eq 1}) 
has a {\it leftmost}\footnote{This terminology means that the repeat
  cannot be extended on the left.  The concept is 
 standard in  the literature, for example   \cite{aldous} and  \cite[p. 19]{AGG90}. }  
$m$-long repeat 
at $(I,J)$ if,
in addition, either $\ell(I)=0$ or 
$$
C_{\ell(I)-1} \neq C_{\ell(J)-1}.
$$
This given, the shotgun edit of coin toss (\ref{eq 1}) is readily defined: make a list
$(I_1,J_1),(I_2,J_2),\dots$ of all the leftmost $n$-long repeats found in
(\ref{eq 1}).  Let
$$
a_i =
\begin{cases}
\overline{C_i} & {\rm if}~ i=r(J_k) ~{\rm some}~ k\\
          C_i  & {\rm otherwise.}
\end{cases}
$$

\subsection{Zero and First Generation Words}  

\noindent Let
$$
C_{0},C_1, \dots, C_{t+n-1}
$$
be a coin toss sequence whose leftmost $n$-tuple repeats occur at $(I_1,J_1),(I_2,J_2),\dots$.
The {\it zero-generation} words of length $h$ 
are simply words of the form:
$$
(C_i,C_{i+1},\dots,C_{i+h-1}).
$$
A {\it first-generation} word is a zero-generation word with exactly one bit
complemented, with the index of the complemented bit required to be $r(J_k)$ for some~$k$:
$$
(C_i,C_{i+1},\dots,\overline{C_{i+j}}, \dots, C_{i+h-1}),  ~~~  i+j = r(J_k).
$$


\subsection{The Good Event $\Gt$}\label{sect good}

We always consider $n$ to be understood, but sometimes we will 
not  
want to
emphasize the role of $t$, hence writing $G \equiv \Gt$. 
Henceforth we shall always assume that $t$ is at most $N=2^n$, since we are interested in cycle lengths
for permutations on a set of size $N$.  Let    $$
C_{0},C_1, \dots, C_{t+n-1}
$$
be a 
length $n+t$
coin toss sequence whose leftmost $n$-repeats occur at $(I_1,J_1),(I_2,J_2),\dots$.
Then the {\it good event} $\Gt$ is defined to be the conjunction of
these six conditions:

\begin{description}

\item[(a)] neither the initial $n$-long word of the coin toss sequence, nor any 
of its $1$-offs\footnote{I.e., words at Hamming distance 1,  hence with our two-letter alphabet, words formed by complementing a single bit.} is repeated (probability of failure $O(tn/N)$);

\item[(b)] all intersections of the form $I_k\cap J_{k^{\prime}}$
are empty (probability of failure $O(t^3n/N^2+tn^3/N)$);

\item[(c)] the sets $I_1,I_2,\dots$ are pairwise disjoint; likewise
$J_1,J_2,\dots$ (probability of failure $O(t^2n^2/N^2+t^3n/N^2+tn^3/N)$);

\item[(d)] no first-generation word of length $n-1$
equals a zero-generation word of length $n-1$, or another first-generation word
of length $n-1$ (probability of failure $O(t^4n^2/N^3+t^3n^3/N^2+tn^3/N)$);

\item[(e)] for every leftmost $(n-1)$-repeat $(I,J)$ we have
$$
r(J_k)\notin I \cup J \cup \{\ell(I)-1,\ell(J)-1\}
$$
for all $k$ (probability of failure $O(t^3n/N^2+t^2n^2/N^2+tn^3/N)$);

\item[(f)] there is no $(2n-1)$-repeat (probability of failure $O(t^2/N^2)$).

\end{description}

\noindent  The indicated probabilities of failure will be proven below in Theorem~\ref{probG}.  First,
though, we will prove a theorem that explains why $G$ is called the ``good event.''

\begin{theorem}\label{theorem detG}\label{detG}
If the coin toss sequence
$$
C_{0},C_1, \dots, C_{t+n-1}
$$
belongs to the good event $G$, then

\begin{description}
\item[Conclusion 1.] The sequentially edited sequence $b_i$ and the shotgun edited sequence $a_i$ agree; and

\item[Conclusion 2.] The sequentially edited sequence $b_i$ and the coin toss sequence have their
leftmost $(n-1)$ repeats at exactly the same positions.
\end{description}
\end{theorem}
These conclusions, along with Theorem \ref{theorem bCrepeats}  in the next subsection, will provide substantial control of the prevalence of $(n-1)$-tuple repeats
\begin{proof}[Proof of Conclusion 1]
Assume, to the contrary, that the $a$ and $b$ sequences differ;
let $i$ be the first position of disagreement:
$$
a_j = b_j, j<i; ~~~  a_i\ne b_i.
$$
There are two possibilities: (1) $a_i\neq b_i$ and $b_i=C_i$; or (2)
$a_i\neq b_i$ and $a_i=C_i$.

\vskip 5pt

\noindent{\bf Case (1).}  Since $a_i\neq C_i$ we have $i=r(J_k)$ for some $k$,
and there is a leftmost $n$-repeat in the $C$ sequence at $(I_k,J_k)$.  But
$a_j=C_j$ for $j\in I_k$ (condition(b)); and $a_j=C_j$ for $j\in J_k\setminus\{i\}$ (condition (c)).
Hence $b_j=a_j=C_j$ for $j \in I_k\cup J_k\setminus\{i\}$.  But $b_i\neq a_i \neq C_i$, so in fact
the $b$-sequence itself has an $n$-repeat at $(I_k,J_k)$.  But the $b$-sequence has the
de Bruijn property, and so the $(I_k,J_k)$ repeat can be backed up $d=\ell(I_k)>0$ steps
to reveal
$$
(b_0,\dots,b_{n-1}) = (b_{i-d-n+1},\dots,b_{i-d}).    ~~~~~~ (d=\ell(I_k)>0)
$$
Since $i-d<i$,
$$
(a_0,\dots,a_{n-1}) = (a_{i-d-n+1},\dots,a_{i-d})
$$
so in fact
\begin{equation}\label{eq Ca}
(C_0,\dots,C_{n-1}) = (a_{i-d-n+1},\dots,a_{i-d}).
\end{equation}
The word on the right side of the last equality is either a zero-generation or a
first-generation word; either case contradicts condition (a).

\vskip 10pt

\noindent{\bf Case (2).}  Because $i$ is a sequential edit point,
($b_i\neq C_i$), it must be that
the $(n-1)$-long word $(b_{i-n+1},\dots, b_{i-1})$ is appearing for a 
second    or later     time, say
$$
(b_{\ell-n+1},\dots, b_{\ell-1}) ~~=~~ (b_{i-n+1},\dots, b_{i-1}), ~~ \ell<i.
$$ 
We must have $b_{\ell}=C_{\ell}$, since no sequential
editing took place at time $\ell$.  (The relevant bit of the feedback logic
had not yet been determined.)   We know that $b_i\neq b_{\ell}$,
else the $b$-sequence contains an $n$-repeat which, as was explained
in Case (1), backs up to yield the contradictory (\ref{eq Ca}).  So,
$C_i\neq b_i \neq b_{\ell} = C_{\ell}$; that is, $C_i=C_{\ell}$ and
$$
(b_{\ell-n+1},\dots, b_{\ell-1}, C_{\ell})
~=~~
(b_{i-n+1},\dots, b_{i-1}, C_i), ~~ \ell<i.
$$ 
Because $i$ is the first point at which the $b$ and $a$ sequences disagree,
\begin{equation}\label{eq 2}
(a_{\ell-n+1},\dots, a_{\ell-1}, C_{\ell})
~=~~
(a_{i-n+1},\dots, a_{i-1}, C_i), ~~ \ell<i.
\end{equation}
Suppose (for the sake of a contradiction) that none of the $a$ bits appearing on
either side of this last Equation~(\ref{eq 2}) was edited by the
shotgun process.  Then we have
\begin{equation}\label{eq 3}
(C_{\ell-n+1},\dots, C_{\ell-1}, C_{\ell})
~=~~
(C_{i-n+1},\dots, C_{i-1}, C_i), ~~ \ell<i.
\end{equation}
We have thus discovered an $n$-long repeat in the coin toss sequence, but it might
not be a {\it leftmost} $n$-long repeat.  So, we look left
to determine the least $m\ge 1$ such that either $\ell-n+1-m<0$ (\textit{i.e.}, you've
gone ``off the board'') or the run of equalities is broken:
$$
C_{\ell-n+1-m} \neq C_{i-n+1-m}.
$$
One of these two will happen
for $m<n$ or else the $C$-sequence is found to contain a
$2n$-repeat, contradicting assumption (f).  But then we have found a leftmost
$n$-repeat in the $C$-sequence beginning at $\ell-n+1-m+1$ and $i-n+1-m+1$;
shotgun editing would consequently modify the $C$-bit at position
$i-n+1-m+1+n-1=i-m+1$.  Since
$$
i-n+1 < i-m+1 \le i,
$$ we have found that one of the $C$-bits on the right side of
Equation~(\ref{eq 3}), namely the one whose index is $i-m+1$, is
changed by shotgun editing, contrary to our earlier supposition that
none of the $a$ bits appearing on either side of the equality~(\ref{eq 2})
was edited by the shotgun process.

\vskip 5pt

\noindent  By condition (c), every $n$-long word in the $a$ sequence either
is a zero-generation word (matches exactly the corresponding $C$-bits) or  
is a first-generation word (matches the corresponding $C$-bits with exactly one change).
Thus, at least one of the $n$-long words appearing in (\ref{eq 2}) is a
first-generation word, and this contradicts condition~(d).
\end{proof}

\begin{proof}[Proof of Conclusion 2.]
We will make use of the $a$ and $b$ sequences
being equal.  Suppose we have a leftmost $(n-1)$ repeat in the coin toss sequence,
\begin{equation}
C_{i+j} = C_{\ell+j}, 0\le j<(n-1);  ~~~{\rm and}~ i=0 ~~{\rm or~}
C_{i-1} \neq C_{\ell-1}.
\end{equation}
By condition (e), none of these $2n$ bits (or $2n-2$ in case $i=0$) can be edited
by the shotgun edit.  Hence, we have a leftmost $(n-1)$-repeat at the same
place in the $a$ sequence, whence also the $b$ sequence.

On the other hand, suppose we have a leftmost $(n-1)$-repeat in the $a$ sequence,
\begin{equation}\label{aRepeat}
(a_i, a_{i+1}, \dots, a_{i+n-2})  =
(a_\ell, a_{\ell+1}, \dots, a_{\ell+n-2}),
\end{equation}
and
$$
 i=0, ~~~{\rm or}~~~ a_{i-1} \neq a_{\ell-1}.
$$
If
$$
(a_{i},\dots,a_{i+n-2}) \neq
(C_{i},\dots,C_{i+n-2})
$$
then $(a_{i},\dots,a_{i+n-2})$ is a first-generation word of length $n-1$ which
equals the first- or zero-generation word $(a_{\ell},\dots,a_{\ell+n-2})$, which is
forbidden by condition~(d).  So, 
\begin{equation}\label{aCiRepeat}
(a_{i},\dots,a_{i+n-2}) =
(C_{i},\dots,C_{i+n-2}).
\end{equation}
Similarly,
\begin{equation}\label{aCjRepeat}
(a_{\ell},\dots,a_{\ell+n-2}) =
(C_{\ell},\dots,C_{\ell+n-2}).
\end{equation}
Altogether by~\eqref{aRepeat},\eqref{aCiRepeat},\eqref{aCjRepeat} we have
\begin{equation}\label{CCRepeat}
(C_{i},\dots,C_{i+n-2}) =
(C_{\ell},\dots,C_{\ell+n-2}).
\end{equation}
If $i=0$, then the last is a leftmost $(n-1)$-repeat in the $C$ sequence, as asserted.
So, to conclude, suppose for the sake of a contradiction
that $i>0$ and that $C_{i-1}= C_{\ell-1}$.
Then we have an $n$-long repeat
$$
(C_{i-1},\dots,C_{i+n-2}) =
(C_{\ell-1},\dots,C_{\ell+n-2}).
$$
Sliding left for $m$ steps, we will encounter a {\it leftmost} $n$-repeat in the
coin toss sequence
$$
(C_{i-1-m},\dots,C_{i+n-2-m}) =
(C_{\ell-1-m},\dots,C_{\ell+n-2-m}),
$$
with $0\le m<n-1$ by condition (f).   But in such a case $a_{\ell+n-2-m}
\neq C_{\ell+n-2-m}$ by the definition of shotgun editing.  However,
for $0\le m<n-1$
$$
\ell \le \ell+n-2-m \le \ell+n-2,
$$
and by~\eqref{aCjRepeat} $a_{\ell+n-2-m}
= C_{\ell+n-2-m}$  The supposition that $i>0$ and that $C_{i-1}= C_{\ell-1}$
has been contradicted, and so~\eqref{CCRepeat} is, indeed, a leftmost $(n-1)$-repeat
as needed.
\end{proof}

\subsection{Probability}\label{sect probability}

In this section we bound the probability of failure of any one of the conditions (a)--(f)
appearing in Theorem~\ref{theorem detG}.   Let $S$ be a set of relations, each of the form
$C_i=C_j$ or $C_i \neq C_j$ with $i<j$.  We assume always that $S$ has at most one relation for
a given $(i,j)$; that is, we don't allow both $C_i=C_j$ and $C_i \neq C_j$.
What is the probability
that a coin toss sequence $C$ will satisfy such a set of relations?
The desired probability is $2^{-|S|}$ {\it provided} the graph associated with
$S$ is cycle free.  The graph we have in mind here is $(V,E)$ where $V$ is the
set $0,1,2,\dots$ and $E$ is the set of pairs $\{i,j\}$ such that at least one (and by
convention exactly one) of the relations $C_i=C_j$ or $C_i \neq C_j$ belongs to $S$.

In particular, if the graph of $S$ consists of the
$n$ pairs $(i,j),(i+1,j+1),\dots,(i+n-1,j+n-1)$ the probability is $2^{-n}=1/N$.
This is quite clear if $I=\{i,\dots,i+n-1\}$ and $J=\{j,\dots,j+n-1\}$ are disjoint,
since then the underlying graph has no vertex of degree 2.  It is also true if
$I$ and $J$ overlap, (of course $I\neq J$): every vertex of degree
$2$ in the graph (\textit{i.e.}, every element of $I\cap J$) has one larger neighbor and one
smaller neighbor.  But a cycle would require at least one vertex with two smaller
neighbors.

We will
have frequent occasion below, in the proof of Theorem~\ref{probG},
to consider sets $S$ whose graphs are the union of 
two such $n$-sets of pairs $(i_1,j_1)$, 
$(i_1+1,j_1+1), \dots, (i_1+n-1,j_1+n-1)$ and
$(i_2,j_2)$, $(i_2+1,j_2+1), \dots$  $(i_2+n-1,j_2+n-1)$.  We begin with a lemma
which shows that in many situations which arise in these proofs
the probability in question is $1/N^2$.

\begin{lemma}\label{lemma lemma}  Let ${\cal G}$ be the graph whose edges
consist of two sets of pairs
$$(i_1,j_1),(i_1+1,j_1+1),\dots,(i_1+m_1-1,j_1+m_1-1)$$
and
$$(i_2,j_2), (i_2+1,j_2+1),\dots,(i_2+m_2-1,j_2+m_2-1).$$  
Then ${\cal G}$ is cycle free if any one of the following three conditions
holds, where we assume $i_1<j_1$ and $i_2<j_2$:
\begin{description}
 
\item{\rm \bf (i)}  $I_1\cap I_2 =\emptyset$

\item{\rm \bf (ii)} $J_1\cap J_2 =\emptyset$

\item{\rm \bf(iii)} $(I_1\cup I_2) \cap (J_1\cup J_2) = \emptyset$, and
$j_2-i_2 \neq j_1-i_1$.
\end{description}
\end{lemma}

\begin{proof}
If $I_1\cap I_2=\emptyset$ then no vertex has two neighbors larger than it. 
If $J_1\cap J_2=\emptyset$ then no vertex has two neighbors smaller than it. 
In case (iii),
all edges out of $I_1\cup I_2$ go to $J_1\cup J_2$, and vice-versa.
A cycle, if there is one, lies within the bipartite graph whose parts
are $I_1\cup I_2$ and  $J_1\cup J_2$, and clearly the cycle must alternate
edges between $(I_1,J_1)$ and $(I_2,J_2)$ types.  If the cycle (of necessity
even in length) uses $\ell$ edges of the first sort and $\ell$ of the 
second, then it has traveled $\ell\times (j_1-i_1)$ in one direction
and $\ell\times (j_2-i_2)$ in the other.  The last part of condition (iii)
makes it impossible for the cycle to have returned to its starting point.
\end{proof}

\begin{theorem}\label{probG}\label{theorem bCrepeats}
Let $G$ be the good event.  Then,
$$
\P(G) ~~ \ge ~~ 1 ~-~ O\left(t^4n^2/N^3+t^3n^3/N^2+tn^3/N\right).
$$
\end{theorem}
\begin{proof}
We shall show that the probability
that a random coin toss sequence of length $t+n$ fails any one of the conditions (a) through (f)
in the definition of $G$ is $O(t^4n^2/N^3+t^3n^3/N^2+tn^3/N)$.  (More explicitly,
each will be shown to fail with the probability indicated in the definition of $G$.)
We invoke the
above Lemma during the proof by citing
 Lemma  (i), Lemma  (ii), 
and Lemma  (iii).

\vskip 10pt

Condition (a): [neither the initial $n$-long word of the coin toss sequence, nor any
of its $1$-offs, is repeated.]  Consider first an exact repetition.
There are $t-1$ places where the repeated sequence
can start, and by earlier remarks the probability that the
second sequence repeats the first is $1/N$.  The same argument applies to the $1$-offs of
the initial pattern, and we conclude that the probability for condition (a) to fail is
less    than $t(n+1)/N$.

\vskip 10pt

Condition (b): [all intersections $I_k\cap J_{k^{\prime}}$ are empty.]  For $k=k^{\prime}$ we bound
the probability of failure by $tn/N$ using the same technique as in case (a).
Suppose that $I_1\cap J_2 \neq\emptyset$.  By the $k=k^{\prime}$ case of the proof
we may assume $J_1$ disjoint from $I_1$ and to its right; and $I_2$ disjoint from
$J_2$ and to its left.  If $I_1\cap I_2=\emptyset$ then Lemma (i) yields the upper
bound $O(t^3n/N^2)$.
If $J_1\cap J_2=\emptyset$ then Lemma (ii) yields $O(t^3n/N^2)$.  In the remaining
case $I_1$ meets $I_2$,
$J_1$ meets $J_2$, and $I_1$ meets $J_2$.   Thus
the union $I_1\cup I_2\cup J_1\cup J_2$ is an interval, and a bound of $O(tn^3/N)$ results.

\vskip 10pt

Condition (c): [the sets $I_1,I_2,\dots$ are pairwise disjoint;
likewise $J_1,J_2,\dots$.]
We will prove the assertion regarding $I_1,I_2,\dots$; the other
assertion is proven in an entirely similar manner.
Suppose $I_1\cap I_2\neq\emptyset$.
We may assume $J_1\cap J_2\neq\emptyset$; otherwise
Lemma (ii) implies an upper bound of $O(t^3n/N^2)$.  So now, both
intersections $I_1\cap I_2$ and $J_1\cap J_2$ are nonempty.  If
any one of the four intersections
$I_a\cap J_b$ is nonempty, then again
the union $I_1\cup I_2\cup J_1 \cup J_2$ is an interval, and we have
the upper bound $O(tn^3/N)$.    So assume
$(I_1\cup I_2) \cap (J_1 \cup J_2) = \emptyset$.
Assume, for the sake of a contradiction, that $r(J_2)-r(I_2)=d=r(J_1)-r(I_1)$.
Then we have $I_1\neq I_2$ and
$J_1\neq J_2$.   Without loss, let us say $I_1$ is left of $I_2$ and
$J_1$ is left of $J_2$.  We have
$C_{\ell(I_2)-1} \neq C_{\ell(J_2)-1}$
because $(I_2,J_2)$ is assumed to be a {\it leftmost} $n$-repeat.  
Since $I_1$ is left of $I_2$, $\ell(I_2)-1\in I_1$; but then
$C_{\ell(I_2)-1} = C_{\ell(J_2)-1}$, the contradiction.  So,
$r(J_2)-r(I_2)=d=r(J_1)-r(I_1)$ is untenable and now Lemma (iii)
implies an upper bound of $O(t^2n^2/N^2)$.

\vskip 10pt

Condition~(d): [no First-generation word of length $n-1$
equals a Zero-generation word of length $n-1$, or another First-generation word
of length $n-1$].  Suppose the first assertion is violated.
Then we have, for some $i$, some $d>0$ and some $j\in \{i,i+1,\dots,i+n-2\}$,
\begin{equation}\label{SetUp}
C_{\ell}=C_{\ell+d} ~~~{\rm for}~ \ell \in \{i,i+1,\dots,i+n-2\}\setminus\{j\},
~{\rm and}~
C_{j}\neq C_{j+d},
\end{equation}
with $r(J_k) \in\{j,j+d\}$ and
with $(I_k,J_k)$ a leftmost $n$-repeat.
Let $I = \{i,i+1,...,i+n-2\}$ and let $J = \{i+d,i+d+1,...,i+d+n-2\}$.
If $I$ and $I_k$ are disjoint then using Lemma (i) the probability is bounded by $O(t^3 \, n/N^2)$.  So assume they intersect, so their union is an interval.  
Similarly we can assume, using Lemma (ii), that $J$ and $J_k$ also intersect so their union is another interval. 
 If these two intervals intersect, forming another interval, we have a probability bound of $O(t \, n^3/N)$.   
 Otherwise $( I \cup I_k) \cap (J \cup J_k)$ is empty.  If $r(J_k) - r(I_k) = d$ then $C_j = C_{j+d}$, contradicting \eqref{SetUp}. So Lemma (iii) implies $O(t^2 \, n^2/N^2)$ for the probability.
This gives an overall bound of $O(t \, n^3/N +t^2 \, n^2/N^2 +t^3 \,n/N^2)$.

Next, suppose the second assertion of (d) is violated.
Then we have, for some $i$, some $d>0$ and some
$j_1,j_2\in \{i,i+1,\dots,i+n-2\}$,
\begin{equation}\label{SetUp2}
(C_i, C_{i+1},\dots, \overline{C_{j_1}}, \dots, C_{i+n-2})
~ = ~~
(C_{i+d}, C_{i+d+1},\dots, \overline{C_{j_2+d}}, \dots, C_{i+d+n-2}),
\end{equation}
with $j_1=r(J_1)$, $j_2+d=r(J_2)$, and $(I_1,J_1), (I_2,J_2)$ leftmost
$n$-repeats.   As reasoned before we have $I\cap J$, $I_1\cap J_1$,
and $I_2\cap J_2$ all empty with probability at least $1-O(tn^2/N)$.  
It follows, from the sheer geometry of the situation, that
$I\cap I_1=\emptyset$.  We may assume that  $I_1\cap I_2=\emptyset$,
since (as proven in (c) above) the probability
of failure is $O(t^3n/N^2+t^2n^2/N^2+tn^3/N)$.
%
We may assume that    $I_1\cap I=\emptyset$,        
for otherwise the union of $I_1$ and $I$ and $J_1$ is a connected interval, 
and then by reasoning as above a bound of   $O( t^3n^3 /N^2)$ results.
We now have all three intersections $I\cap I_1$, $I\cap I_2$, and $I_1\cap I_2$
being empty; and by an obvious embellishment of Lemma (i) 
the probability of the remaining case is $O( t^4n^2 /N^3)$.

\vskip 10pt

Condition (e): [for every leftmost $(n-1)$-repeat $(I,J)$ we have
$$
r(J_k)\notin I \cup J \cup \{\ell(I)-1,\ell(J)-1\}
$$
for all $k$.]  Say $I=\{i,i+1,\dots,i+n-2\}$
and $J=\{i+d,i+d+1,\dots,i+d+n-2\}$. 
The probability that $I_k\cap J_k$ is not empty is $O(tn/N)$, so
assume $I_k\cap J_k=\emptyset$.
If $r(J_k) \in I\cup\{i-1\}$, then, since $I_k$ lies entirely to 
the left of $J_k$, $I_k\cap I=\emptyset$.  By Lemma (i)
the probability of this is $O(t^3n/N^2)$.  

The probability that $I\cap J$ is not empty is $O(tn/N)$, so
assume both $I_k\cap J_k$ and $I\cap J$ are empty.  The probability
that $I_k\cap I$ is empty is, by Lemma (i), $O(t^3n/N^2)$, so assume
$I_k\cap I\neq\emptyset$.
If $I\cap J_k$ or $I_k\cap J$ is nonempty then  
$I_k\cup I \cup J_k \cup J$ is an interval, and the probability of this
is $O(tn^3/N)$.  So, assume
$(I\cup I_k) \cap (I_k\cup J) = \emptyset$.
If $r(J_k)-r(I_k)=r(J)-r(I)$, then
\begin{align*}
C_{\ell(I)-1} ~~ =    ~~ C_{\ell(J)-1}  &  ~~~{\rm by}~  \ell(I)-1 \in I_k \\
C_{\ell(I)-1} ~~ \neq ~~ C_{\ell(J)-1}  &  ~~~{\rm by}~  (I,J) ~{\rm being~a~{\it leftmost}~} (n-1)-{\rm repeat},
\end{align*}
an impossibility.  So, $r(J_k)-r(I_k)\neq r(J)-r(I)$ and Lemma (iii) gives the
bound $O(t^2n^2/N^2)$ for this final scenario.

\vskip 10pt

Condition (f):  [there is no $(2n-1)$-repeat.]  Easily, the failure probability
is at most $2t^2/N^2$.
\end{proof}

\subsection{Coin Tossing Versus Paths in the de Bruijn Graph}
\label{sect induced probability}

\begin{theorem}\label{theorem honest t}
Let $b_0,\dots,b_{t+n-1}$ be a bit string.  Then, the probability that this string arose by
the sequential editing of an $n+t$ long coin toss sequence is the same as the probability
that it arose by choosing a logic $f:V^{n-1}\rightarrow V$ and starting position
$(b_0,\dots,b_{n-1})$ each uniformly at random.
\end{theorem}
\begin{proof}
Without loss of generality we assume the given string 
has the de Bruijn property.  (Else, the two probabilities
are both zero.)   First, let's compute the probability that $b$ arose by sequential 
editing of a $(t+n)$-long coin toss.  The probability of the coin toss yielding
$b_0,\dots,b_{n-1}$ is $(1/2)^n$.  Consider
$b_i$, with $i\ge n$.   If $(b_{i-n+1},\dots,b_{i-1})$ is equal to $(b_{j-n+1},\dots,b_{j-1})$
for some $j$ in the range $n\le j<i$, then sequential editing
says to let $b_i$ be what it ought to be: $b_{i-n}\oplus f(b_{i-n+1},\dots,b_{i-1})$.  In which
case, it does not matter what value $C_i$ has.   But if $i\ge n$ and $(b_{i-n+1},\dots,b_{i-1})$
has not been seen before (among $(n-1)$-long words ending at a position greater then or equal
to $n$), then $C_i$ must be equal to $b_i$.   (And, we remember henceforward
the value of
$f(b_{i-n+1},\dots,b_{i-1})$ is $b_i\oplus b_{i-n}$.)   Altogether, then, the probability that a
length $t+n$ coin toss will yield a given sequence
$b_0,b_1,\dots,b_{t+n-1}$ by sequential editing is
$2^{-r}$ where
$$ r =  n-1+ \#{\rm distinct~} (n-1){\rm -long~subwords~ending~at~position~} n-1 {\rm ~or~later}.
$$

Now let's compute the probability that $b$ arose by choosing a starting position and
logic at random.
Classify each position $i$, $0\le i\le t+n-1$, as Type I or Type II.  
The position is Type I if $i\ge n$ and the preceding $(n-1)$ long word $(b_{i-n+1},\dots,b_{i-1})$
is appearing for the first time in the $b$-sequence. 
The position is Type II otherwise:
either $i<n$, or the preceding $(n-1)$ long word $(b_{i-n+1},\dots,b_{i-1})$ is appearing for the 
second or later time.  It should be clear that the probability in question is
$$
\left(\frac{1}{2}\right)^{n+\#{\rm Type~I}}.
$$
The two probabilities just calculated agree.\footnote{There are several
  interesting results in Maurer \cite{Maurer} for cycles in de Bruijn
  graphs; one must be careful to think about the factor $2^{\pm r}$ in 
going back and forth between these estimates, and estimates for a
random $\pi_f$, corresponding to \emph{randomly resolved} de Bruijn graphs.}
\end{proof}

\subsection{Notation for Paths Starting at $k$ Random $n$-tuples}

We now
fix $k \ge 1$ and use the notation $e_1,\ldots,e_k$ to name $k$
random
$n$-tuples.   Collectively, these $k$ edges of $D_{n-1}$ are denoted
\begin{equation}\label{def be}
   \be = (e_1,e_2,\ldots,e_k)  \in (\mathbb{F}_2^n)^k.
\end{equation} 
Picking a random feedback $f$, and $k$ random $n$-tuples, independent
of $f$, is equivalent to picking one element, uniformly at random from
the space
\begin{equation}\label{def S_{n,k}}
    S_{n,k} := \{ (f,\be)\!:  \ f:  \mathbb{F}_2^{n-1} \to \mathbb{F}_2,
 \be \in (\mathbb{F}_2^n)^k  \}, \text{ with } |S_{n,k}|=2^{2^{n-1}+kn}.
\end{equation}

The choice of $(f,\be)$ from $S_{n,k}$ determines $k$ infinite
periodic sequences of edges:  for $a=1$ to $k$,
\begin{equation}\label{def k segments}
   \seg(f,e_a) := (e_{a,0} e_{a,1} e_{a,2}\cdots) \emph{ where }
   e_{a,0}=e_a, \text{ and for } i \ge 0, e_{a,i+1}= \pi_f(e_{a,i}).
\end{equation}

For the sake of comparison with coin tossing, we often look at such
paths only up to time $t$ (this is what motivated our terminology
\emph{segment}): 
\begin{equation}\label{def k seg t}
 \text{ for } a=1 \text{ to } k, \ \   \seg(f,e_a,t) = (e_{a,0} e_{a,1} \cdots e_{a,t}).
\end{equation}

\subsection{$(k,t)$-sequential Editing}

Now we will define a modification of the sequential editing process
that was
discussed earlier in Section~\ref{sect sequential}.   The reader should
bear in mind our ultimate goal. We wish to study what happens when
a feedback logic $f$ is chosen at random; $k$ different starting
$n$-tuples $e_1,\dots,e_k$ are chosen at random; and $k$ walks of
length $t$ are generated, the first starting from $e_1$ and using
the logic $f$ to continue for $t$ steps; the second starting from
$e_2$, etc.  As in Section~\ref{sect sequential}, we
wish to generate these walks using $k(n+t)$ coin tosses, and we
would like to have  an analog to Theorem~\ref{theorem honest t}
saying that our procedure for passing from the coin toss to the
$k$ walks perfectly simulates the process of choosing a logic and
starting points at random.   The reader can almost certainly
envision the natural way to achieve this, but we will write out
the details.

The first $n+t$ coins are used exactly as in Section
\ref{sect sequential}: Rule 1 is applied to the first $n$
coin tosses to yield starting point $e_1$, and then
Rule 2 is applied $t$ times to get the overlapping
$n$-tuples
$e_1=e_{1,0}, e_{1,1}, \ldots, e_{1,t}$ that form the
first walk. 
Equivalently,
this segment is spelled out by the $(n+t)$ de Bruijn bits $b_0 \ldots b_{t+n-1}$,
and along the way, some feedback logic bits have been defined.

Then, for the next $n$ coin tosses,  $C_i$ for $i=t+n$ to $i=t+2n-1$
inclusive, sequential editing is \emph{suspended};
again Rule 1 is applied, to give 
$$
  e_2 := (b_{t+n},\ldots,b_{t+2n-1}) :=
  (C_{t+n},\ldots,C_{t+2n-1}),
$$
with no new feedback logic bits learned.
Then, Rule 2 is applied for the next $t$ input bits, $C_i$ for
$i=t+2n$ to $i=2t+2n-1$ to create the second walk of
length $t$,  $\seg(f,e_2,t)$  --- remembering of course
those feedback logic bits that were learned during the creation of 
$\seg(f,e_1,t)$, and (most likely) learning some  new feedback logic
bits in the process.  (It might be the case that
  $e_2=e_1$, or that $e_2$ appears in the first walk, in
which case, we don't learn any new feedback logic bits.)
If $k>2$, we continue in a
similar fashion, first suspending editing for time $n$, during
which time we learn no
new  feedback logic bits  and we form
$e_a := (b_{(a-1)(t+n)},\ldots,b_{(a-1)t+an-1}) :=
(C_{(a-1)(t+n)},\ldots,C_{(a-1)t+an-1})$,  then returning to Rule 2
for the next $t$ bits, to fill out $\seg(f,e_a,t)$.

For $k,t \ge 1$ we define
\begin{eqnarray}\label{def ktsedit}
   \ktsedit: \{0,1\}^{k(n+t)} & \to & (\mathbb{F}_2^{t+n})^k \\ 
          (C_0,C_1,\ldots,C_{k(n+t)-1}) & \mapsto &
   (\seg(f,e_1,t),\ldots,\seg(f,e_k,t) )
\end{eqnarray}
as given by the above procedure. 

It may, or should,  seem intuitively obvious that $\ktsedit$,  applied to
an input uniformly chosen from $ \{0,1\}^{k(n+t)}$, induces the same
distribution on the $k$ segments of length $t$ in 
\eqref{def k segments}, as does a uniform pick from $S_{n,k}$ and
iteration of $\pi_f$ from each of $e_1,\ldots,e_k$.  We claim that
the argument given in the
proof of Theorem~\ref{theorem honest t} can be adapted to show this.

\subsection{The Good Event $\Gkt$ for $(k,t)$-sequential Editing}
There are two \emph{different} ways to produce $k$ walks each of
length $t$ out of a sequence of $k(n+t)$ coin tosses.
The first, with $t'=(k-1)n+kt$ playing the role of $t$,
is simple sequential edit, to determine a
starting
$n$-tuple $e$, and one path $e_0,e_1,\ldots,e_{t'}$ corresponding to
$t'=(k-1)n+kt$
iterates of $\pi_f$ starting from $e$.  The good event, regarding this
first procedure, is really $G \equiv G_{((k-1)n+kt)}$.
We can then \emph{cut} the path of length $t'$ to produce $k$ paths of
length $t$;  see~\eqref{segment vertex} to see the natural notation
associated with such cutting.
The second procedure is is to apply $\ktsedit$, defined in the
previous section, to produce a
$k$-tuple of starting edges, $\be$, and $k$ segments of length $t$, as
in~\eqref{def k seg t}.  The good event, regarding this second
procedure,
to be called $\Gkt$, is designed so that the two procedures agree.  We
simply
take all of the demands of the good event for simple editing on
$k(n+t)$ coins, and throw in additional requirements to ensure the
\emph{suspensions} of editing involved in the definition of
$\ktsedit$.   Informally, these additional requirements are that every
$(n-1)$ tuple which appears at some time $j$ involved in \emph{suspension}
occurs at no other time $i$ in the coin toss sequence.  Formally,  given
$n,k,t$,
the \emph{bad} event $B$ is given by
\begin{equation}\label{def bad}
 B =   \bigcup_{ i \in [0, k(n+t)-n+1]}  \ \ \bigcup_{j \in
     \cup_{a=1}^{k-1} [a(n+t)-n+2,a(n+t)]}  M_{ij}
\end{equation}
where the event $M_{ij} = \emptyset$ if $i=j$, and otherwise
$$
   M_{ij} = \{d_{\rm HAMMING}(C_iC_{i+1}\cdots
   C_{i+n-2},C_jC_{j+1}\cdots C_{j+n-2}) \le 1 \},
$$
and the good event is then
\begin{equation}\label{def Gkt}
\Gkt = G_{((k-1)n+kt)} \setminus B.
\end{equation}
Since a word of length $n-1$ has $n$ neighbors at Hamming distance 1
or less, $\p(M_{ij}) = n/2^{n-1}$ for $i \ne j$, so
that $\p(B) \le (n+t)k^2n^2 \times 2/N$, for the sake of extending
Theorem~\ref{probG}.

We now consider the following to have been proved; it is a single
theorem, to give the extensions of Theorems \ref{detG} and \ref{probG}
and
\ref{theorem honest t},  
appropriate to $k,t$ sequential editing.  Note that  in the final conclusion of Theorem \ref{kt theorem} we treat $k$ as fixed  while $n,t \to \infty$, so that $t$ and $kt$ are of the same order, and we take the assumption  $t/\sqrt{N} \to \infty$ 
so that the three terms in the bound from Theorem \ref{probG} are covered by a single term.

\begin{theorem}\label{kt theorem}
\begin{enumerate}[(i)]
\item     The procedure $\ktsedit$, applied to a coin toss sequence $$
(C_0,C_1,\ldots,C_{k(n+t)-1}) $$
chosen uniformly at
random from $\{0,1\}^{k(n+t)}$, yields $k$ segments of length $t$, 
$(\seg(f,e_1,t),\ldots,\seg(f,e_k,t) )$ with \emph{exactly} the same
distribution as obtained by a random feedback logic $f$ and $k$
starting $n$-tuples,  $\be = (e_1,\ldots,e_k)$.  
\item  The good event
$\Gkt \subset \{0,1\}^{k(n+t)}$, defined by~\eqref{def Gkt} --- which
ultimately involves conditions {\rm (}a{\rm )} through {\rm (}f{\rm )} from 
Section~\ref{sect good},
applied with $t'=(k-1)n+kt$ in the role of $t$, is such that for every
outcome in $\Gkt$, the bit sequence $b_0b_1 \cdots b_{k(n+t)-1}$
{\rm (}and the equivalent sequence of overlapping $n$-tuples, $e_0e_1\cdots
e_{(k-1)n+kt}${\rm )}
formed by single sequential edit agrees with the shotgun edit of the
$k(n+t)$ coins, and leftmost $(n-1)$-tuple repeats have the same
locations in $b_0b_1 \cdots b_{k(n+t)-1}$ and in
$(C_0,C_1,\ldots,C_{k(n+t)-1})$.
\item Also, on the good event $\Gkt$, the $k$ segments of length $t$,
produced by $\ktsedit$ and 
notated as in~\eqref{def k seg t} match exactly with $e_0\cdots
e_t, e_{t+n}\cdots e_{2t+n}, \ldots, e_{(k-1)(t+n)} \cdots
e_{(k-1)n+kt}$,
produced by cutting the output of the single sequential edit of $k(n+t)$ coins.
\item Finally, if $t/\sqrt{N} \to \infty$ with $k$ fixed, then $\p(\Gkt) \ge 1 - O(n^3 \, t^3/N^2)$.
\end{enumerate}
\end{theorem}

We summarize:  there is an exact operation, sequential editing of
$n+t$ coin tosses, which achieves the exact distribution of
$\seg(f,e,t)$,
as induced by a uniform choice of $(f,e)$ from its $2^{2^{n-1}} 2^n$
possible values, followed by starting at $e$ and taking $t$ iterates
of the
permutation
$\pi_f$.  There is a good event $G \equiv G_{t}$, with $\p(G) \to 1$ provided that
$t^3n^3/N^2 \to 0$, for which the sequential edit agrees with the
shotgun edit, and $v_i=v_j$ iff the coins have a leftmost
$(n-1)$-tuple repeat at $(i,j)$.  This sequential edit can be used
with $k(n+t)$ in place of $n+t$, to create one long segment; there is
the corresponding good event $G_{t^{\prime}}$, $t^{\prime}=(k-1)n+kt$.
There is a second, distinct operation, $\ktsedit$,
for editing $k(n+t)$ coin tosses,
to yield the exact distribution of $k$ segments of length $t$ under a
single logic $f$ and $k$ starting $n$-tuples, $\be =
(e_1,\ldots,e_k)$; that is, the distribution of
$(\seg(f,e_1,t),\ldots,\seg(f,e_k,t))$ as induced 
by a uniform choice of $(f,\e)$ from its $2^{2^{n-1}} 2^{kn}$
possible values.
And there is a corresponding good event $\Gkt \subset  G_{t^{\prime}}$,
with
$$ 
\p(G_{t^{\prime}} \setminus \Gkt) \le 2k^2n^2(n+t)/N,
$$
formed by adding the constraint that $i$ or $j \in \cup_{0 \le a <k}
[a(n+t)-n+2,a(n+t)]$ implies that there is not an $(n-1)$ tuple repeat
at $(i,j)$.   On the event $\Gkt$, the $k$-sequential edit agrees
exactly with the \emph{cutting} of $\seg(f,e_1,k(n+t)-n)$.

\subsection{A Cutting Example}\label{sect cutting}

We now illustrate some of the concepts just introduced, with an example
and with Figures 1-4.
Take $n=10,t=90,k=3$.  So, to generate $k=3$ segments of length
$t=90$, we start with $k(n+t)=300$ coin tosses, used to generate one
segment of length $k(n+t)-n=290$. When we have in mind a single
segment of length $t$,  we will use a single subscript to label the
edges, so that with $e=e_0$, the segment is a list of $t+1$ edges
\begin{equation}\label{segment 1}
\seg(f,e,t)= e_0 e_1 \cdots e_t.
\end{equation}

 The coin tosses, indexed from
$i=0$ to $i=299$, are labeled $C_i$, the de Bruijn bits formed
by
sequential edit are labeled $b_i$, and the bits formed by shotgun
edit are labeled $a_i$.  On the good event $G$, we will have
$a_i=b_i$ for all $i$.  The vertex $v_i$ is the $(n-1)$- tuple of bits
starting with $b_i$, the edge $e_i$  is the $n$-tuple of bits starting 
with $b_i$, and edge $e_i$ at time $i$ goes from vertex
$v_{i}$ to $v_{i+1}$:
$$  
   v_{i} = b_{i}b_{i+1}\cdots b_{i+n-2}, \  e_i = b_{i}b_{i+1}\cdots b_{i+n-1}, \
e_i= (v_{i},v_{i+1})  .
$$
We also view the segment in~\eqref{segment 1}
as a list of $t+2$ vertices, or as
a list of $t+n$ bits, and abuse notation by writing equality, so that 
\begin{equation}\label{segment 1 vertex}
\seg(f,e,t)=  v_0 v_1 \cdots v_t v_{t+1}.
\end{equation}
\begin{equation}\label{segment 1 bit}
\seg(f,e,t)= b_0 b_1\cdots b_{n-1}b_n \cdots b_tb_{t+1}\cdots b_{t+n-1}.
\end{equation}

Since we are particularly interested in leftmost $(n-1)$-tuple repeats, we shall
suppose that we are in the good event $G$, and the leftmost
$(n-1)$-tuple repeats in the coin-toss sequence 
are at (56,153), (120,260), and (135,175).  Thanks to $G$ occurring,
we know that all 291 edges $e_0$ to $e_{290}$ are distinct, and the
only vertex repetitions are  $v_{56}=v_{153}, v_{120}=v_{260}$, and
$v_{135}=v_{175}$. 
One way of indicating where these vertex repeats occur is to draw some
auxiliary lines pointing to the locations, as in Figure~\vref{fig 1}.
Figure~\vref{fig 2} gives a two-dimensional (``spatial'') view of the same situation.

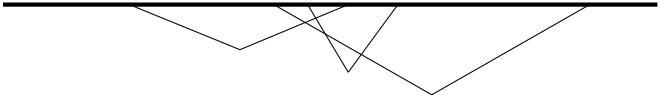
\begin{figure}
   \begin{tikzpicture}[scale=.03]
    \draw[ultra thick] (0,0) -- (290,0);  
    \draw (56,0) --(105,-20);\draw (153,0) --(105,-20);
    \draw (120,0) --(190,-40);\draw (260,0) --(190,-40);
    \draw (135,0) --(153,-30);\draw (175,0) --(153,-30);
 
   \end{tikzpicture}
   \caption{An example, one segment of length 290, where there are
     three leftmost $(n-1)$-tuple repeats, at (56,153), (120,260), and
     (135,175)}\label{fig 1}
\end{figure}
 

\begin{figure}
   \begin{tikzpicture}[scale=.03]
   \draw[ultra thick] (0,0) -- (290,0);  
   \draw (56,0) --(105,-20);\draw (153,0) --(105,-20);
   \draw (120,0) --(190,-40);\draw (260,0) --(190,-40);
   \draw (135,0) --(153,-30);\draw (175,0) --(153,-30);
   \draw[ultra thick] (0,0) -- (0,290);
   
   \draw (0,56) --(-20,105);\draw (0,153) --(-20,105);
   \draw (0,120) --(-40,190);\draw (0,260) --(-40,190);
   \draw (0,135) --(-30,153);\draw (0,175) --(-30,153);

   \draw [dotted] (0,0) -- (290,290);    
   \draw [dotted] (0,290) -- (290,290);    

   \draw [fill] (56,153) circle [radius=3];
   \draw [fill] (120,260) circle [radius=3];
   \draw [fill] (135,175) circle [radius=3];

   \end{tikzpicture}
   \caption{The same example: one segment of length 290, where there
     are three leftmost $(n-1)$-tuple repeats, at locations (56,153),
     (120,260), and (135,175).  Now, the locations are plotted in
     standard Cartesian coordinates.}\label{fig 2}
\end{figure}
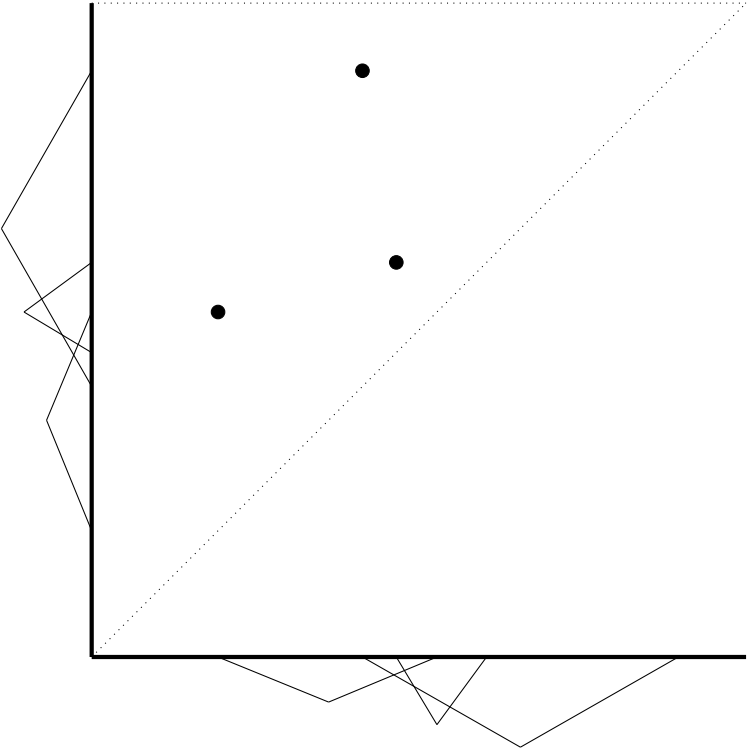
 

When we cut the single long segment in~\eqref{segment 1} into $k=3$
segments,
we use two indices;  the first runs from 1 to $k$, and the second runs
from 0 to $t$.  Including the relation with~\eqref{segment 1}, for
Example 1, but with the labels $e_1,\ldots,e_k$ overloaded --- since
they also appear on the left side, naming the starting edges for
the $k$ segments --- 
this will give
$$
\begin{array}{ccccc}
\seg(f,e_1,t) & = & e_{1,0}\cdots e_{1,90} & = & e_0\cdots e_{90}
\\
\seg(f,e_2,t) & = & e_{2,0}\cdots e_{2,90} & = & e_{100}\cdots e_{190}
\\
\seg(f,e_3,t) & = & e_{3,0}\cdots e_{3,90} & = & e_{200}\cdots e_{290}.
\end{array}
$$

The same $k=3$ segments of length $t=90$, presented as lists of
vertices (which here are 9-tuples) are notated as
\begin{equation}\label{segment vertex}
\begin{array}{ccccc}
\seg(f,e_1,t) & = & v_{1,0}v_{1,1}\cdots v_{1,90}v_{1,91} & = & v_0v_1\cdots v_{90}v_{91}
\\
\seg(f,e_2,t) & = & v_{2,0}v_{2,1}\cdots v_{2,90}v_{2,91} & = & v_{100}v_{101}\cdots v_{190}v_{191}
\\
\seg(f,e_3,t) & = & v_{3,0}v_{3,1}\cdots v_{3,90} v_{3,91} & = & v_{200}v_{201}\cdots v_{290}v_{291}.
\end{array}
\end{equation}

Collectively, these $k$ segments are given by a deterministic function
of
$(f,\be)$,  where $\be=(e_1,e_2,\ldots,e_k)$ names all $k$ starting points.

\subsection{Coloring} \label{sect color}

Imagine the $k$ segments of length $t$ as pieces of (directed) yarn,
with $k$ different ``primary'' colors. Vertices that appear only once
get the primary color of the segment they come from; vertices that
appear twice on the same segment might be visualized as having a more
saturated version of the primary color of that segment.  The
interesting case occurs when a vertex appears on two different
segments; such a vertex, call it $\vs$, gets each of two primary
colors --- and its secondary color shows which two segments this
vertex lies on; for example imagine that the two strands are red and
yellow, so that $\vs$ is colored orange.    Figure \ref{fig 3} on page \pageref{fig 3} and Figure \vref{fig 4} 
 illustrate this coloring.


\begin{figure}
   \begin{tikzpicture}[scale=.04]
   \draw  (0,0) -- (290,0);  
   \draw[fill=red,ultra thick,red] (0,0) -- (90,0);
   \draw[fill=yellow,ultra thick,yellow] (100,0) -- (190,0);
   \draw[fill=blue,ultra thick,blue] (200,0) -- (290,0);
   \draw (56,0) --(105,-20);\draw (153,0) --(105,-20);
   \draw (120,0) --(190,-40);\draw (260,0) --(190,-40);
   \draw (135,0) --(153,-30);\draw (175,0) --(153,-30);
 
   \end{tikzpicture}
   \caption{Coloring.  An example, with $k=3$, $n=10,t=90$.  The same
     one segment of length 290, as in Figure~\vref{fig 1}, where there
     are three leftmost $(n-1)$-tuple repeats, at (56,153), (120,260),
     and (135,175).  Now the first segment is colored red, the second
     yellow, and the third blue.}\label{fig 3}  
\end{figure}
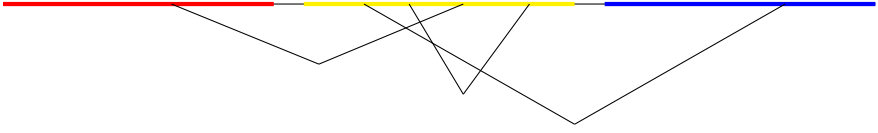

\begin{figure}
   \begin{tikzpicture}[scale=.03]
  
   \draw  (0,0) -- (290,0);  
   \draw[fill=red,ultra thick,red] (0,0) -- (90,0);
   \draw[fill=yellow,ultra thick,yellow] (100,0) -- (190,0);
   \draw[fill=blue,ultra thick,blue] (200,0) -- (290,0);
   \draw (56,0) --(105,-20);\draw (153,0) --(105,-20);
   \draw (120,0) --(190,-40);\draw (260,0) --(190,-40);
   \draw (135,0) --(153,-30);\draw (175,0) --(153,-30);
   \draw (0,0) -- (0,290);
   \draw[fill=red,ultra thick,red] (0,0) -- (0,90);
   \draw[fill=yellow,ultra thick,yellow] (0,100) -- (0,190);
   \draw[fill=blue,ultra thick,blue] (0,200) -- (0,290);
   \draw (0,56) --(-20,105);\draw (0,153) --(-20,105);
   \draw (0,120) --(-40,190);\draw (0,260) --(-40,190);
   \draw (0,135) --(-30,153);\draw (0,175) --(-30,153);

   \draw [dotted] (0,0) -- (290,290);    
   \draw [dotted] (0,290) -- (290,290);    
   \draw [dotted] (0,100) -- (90,190); 
   \draw [dotted] (0,200) -- (90,290);
   \draw [dotted] (100,200) -- (190,290); 
   
\draw [fill=orange,orange] (56,153) circle [radius=3];
   \draw [fill=green,green] (120,260) circle [radius=3];
   \draw [fill=yellow,yellow] (135,175) circle [radius=3];

   \draw [red] (0,190)--(90,190);
   \draw [red] (0,200)--(90,200); 
   \draw [red] (0,290)--(90,290); 

   \draw [red] (0,100)--(90,100);
   \draw [yellow] (100,200)--(190,200); 
   \draw [yellow] (90,100)--(90,190);  
   \draw [blue] (90,200)--(90,290);
   \draw [yellow] (100,290)--(190,290);
   \draw [blue] (100,200)--(100,290);
   \draw [blue] (190,200)--(190,290);

   \end{tikzpicture}
   \caption{Coloring and cutting;  a succinct way to visualize both.
     The $k$ segments of length $t$ are still shown as they appear
     along
the single segment of length $k(t+n)-n$.   We also show the ${k
  \choose 2}$ $t$ by $t$ squares where matches may occur between two
differently colored length $t$ segments.  Note the repeat at (56,153)
is a vertex colored both red and yellow, hence orange.  The repeat at
(120,260) is a vertex colored both yellow and blue, hence green.  The
vertex at (135,175) is colored yellow twice - we could show
it as an extra-saturated yellow but did not. The significance of the 
diagonals of the
small squares is explained in Section \ref{sect earliest}.} \label{fig 4}
\end{figure}
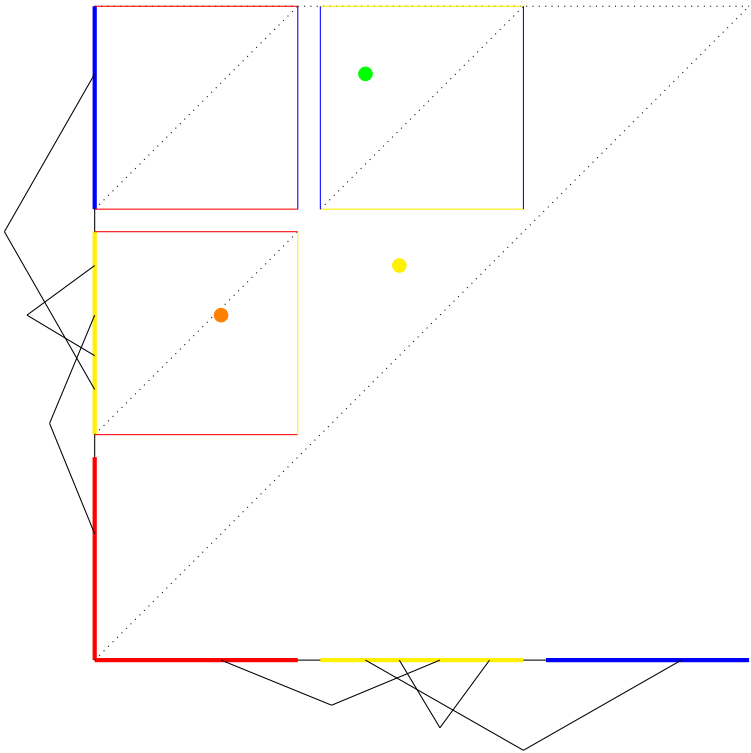


\section{Toggling}\label{sect toggling}

To \emph{toggle} a logic $ f:  \mathbb{F}_2^{n-1} \to \mathbb{F}_2$
 at a vertex $v \in  \mathbb{F}_2^{n-1}$
is simply to get a new $f$ from the old, by changing the value at $v$.
This is called a ``cross-join step'', and is studied extensively in
the context of 
cycle joining algorithms to create a full cycle logic.  Our interest
in toggling is different:  we have $k \ge 2$  segments induced by a
logic $f$ and $k$ starting $n$-tuples, $e_1,\ldots,e_k$, and we want
to choose $m$ different ``toggle points'' in the role of $v$,
to get a \emph{nice} family of $2^m$ related logics.  All this is done
in the interest of showing that the chance that $e_1$ and $e_2$ lie on
the same cycle of $\pi_f$ is approximately one-half, for large $n$,
and more generally, that the chance $e_1,\ldots,e_k$ all lie on the
same cycle is approximately $1/k$, and even more, that the permutation
$\pi_f$, relativized to $e_1,\ldots,e_k$, is approximately uniformly
distributed over all $k!$ permutations.  This introductory paragraph is
intentionally short and vague; the full details use all of 
Sections \ref{sect coins} -- \ref{sect together}. 
Section \ref{sect big} gives a longer attempt at introduction,
including
Figure~\vref{fig 5}, showing the huge collection of candidate toggle
vertices, using $k$ \emph{colors} to help visualize the $k$
segments of interest. 

\subsection{Big Picture Perspective:  $k$ Colored Segments, $m$ Toggle
Points}\label{sect big}

We will have $k$ segments each of length $t=N^{.6}$.   The expected
number of leftmost $(n-1)$-tuple repeats within a single segment is about ${t \choose 2}/N \doteq
.5 N^{.2}$.
The expected number of
repeats between two different given segments is about $t^2 / N = N^{.2}$, 
so the expected number of repeats between two different segments,
combined over all 
${k \choose 2}$
choices for which two segments, is about ${k \choose
  2} \times N^{.2}$.   This is a huge number of repeats (each based on
one vertex having a secondary coloring), and we intend to find $m$
such repeats, say at $\vs_1,\ldots,\vs_m$ in  narrowly constrained spatial positions.  The goal is
to show that, with high probability, for all $2^m$ choices of how to
change $f$ by toggling the values of $f(\vs_i)$ for $i \in I \subset
\{1,2,\ldots,m\}$, the \emph{same} $m$ vertices will be picked out by
the 
narrow two-dimensional spatial constraints.

In this section we present further figures intended to assist the
reader's intuition.  We also give an algorithm which for a given logic
$f$ and starting edges $e_i$ finds $m$ vertices $\vs_1$,
\dots,$\vs_m$.  These points---we call them toggle points---give rise
to a family of $2^m$ functions.  We also define 
(in Section \ref{sect tangle}) a process 
called {\it
  relativisation} which associates with $\pi_f$ in $S_N$ a permutation
$\sigma$ in $S_k$, $k$ being fixed and $N \to \infty$.  It will be
shown that as $f$ varies over the $2^m$ functions in a ``toggle
class'', the resulting $\sigma's$ cover $S_k$ almost uniformly.  In
Section~\ref{sect sampling}, it will be shown that this uniform
coverage of $S_k$ (for each fixed $k$) is a sufficient condition to
prove Theorem~\ref{theorem 1}.

A critical issue is that the algorithm for choosing the toggle points
must be such that, if the feedback $f$ is replaced by any one of the $2^m$ 
functions in the toggle class, the algorithm would find the same toggle
points and the same class.\footnote{Consider the
simplest situation, $k=2$
and $m=1$, where one is trying to prove  \eqref{thm 1 E}
by showing that $\p(e_1,e_2$ lie on the same cycle) is approximately
one half. 
Knowing that the segments starting from $e_1$ and $e_2$ have high
probability of reaching a common vertex $\vs$, 
and that performing a cross-join step by toggling the
logic $f$ at this $\vs$, to get a new logic $\fs$, \emph{changes} whether or not
$e_1$ and $e_2$ lie on the same cycle, one might consider the proof
complete.
The fallacy is that this procedure does not \emph{pair up} $f$ with
$\fs$, \textit{i.e.}, it need not be the case that $(\fs)^* = f$, because the
procedure used to find $\vs$ (from $f$, given $e_1,e_2$) might find a
different $v$ when applied to $\fs$.  Overcoming this fallacy entails
the study of \emph{displacements}, starting  in 
Sections \ref{sect sub toggle} -- \ref{sect displace}.
\label{footnote fallacy} }  
However this is not necessarily the case
and the probability of success for the algorithm must be estimated;
this leads to the definition of an event~$H$.
\begin{figure}
\begin{center}
\includegraphics{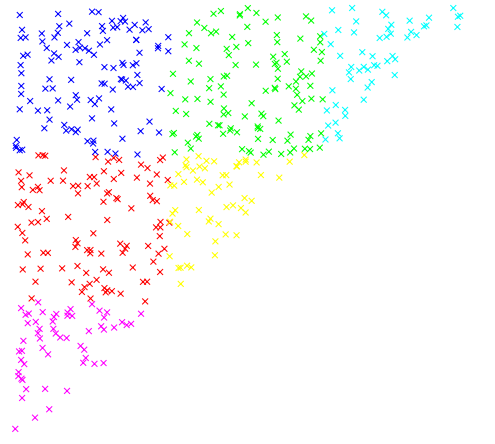}
\end{center}
\caption{Take $n=34$, $N=2^n$, and $t=3 \times (n+N^{.6})-n$.  The
  expected number of leftmost $(n-1)$-tuple repeats is about ${t
    \choose 2}/N \doteq 501.4$.  The picture shows 500 ``arrival''
  points, giving the locations of repeats, plotted as for one segment
  of long length $t$.  In each $N^{.6}$ by $N^{.6}$ square, the
  expected number of points is $N^{.2} \doteq 111.4$.  The color scheme is intended to be
purple, green, blue across the top row, orange, yellow for the middle, and red (magenta) for the bottom.}\label{fig 5}
\end{figure}

\begin{figure}
\begin{center}
\includegraphics{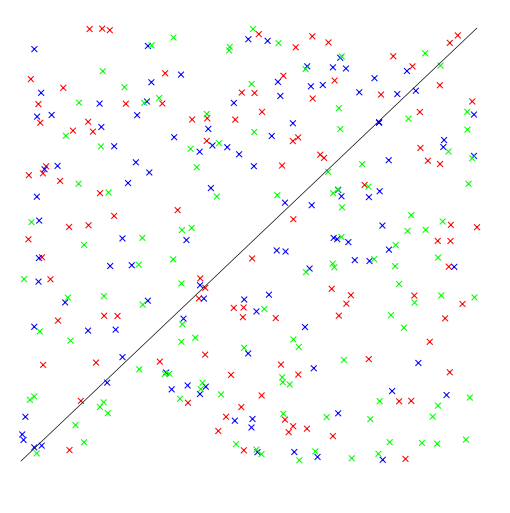}
\end{center}
\caption{About 333 of the 500 occurrences of repeats from
  Figure~\vref{fig 5}, but now viewed as among $k=3$ segments of
  length $t=N^{.6} \doteq 1.38 \times 10^6$. The ${k \choose 2}$ $t$
  by $t$ above-diagonal squares from Figure~\ref{fig 5} are
  superimposed, so the expected number of points is about ${k \choose 2}
  \times N^{.2} \doteq 334.3$. The approximately 167.1 repeats where
  both occurrences lie in the same segment, corresponding to the $k$
  right triangles hugging the diagonal in Figure~\ref{fig 5}, are not
  shown.  In Section \ref{sect scale} we discuss this picture,
  suggesting scaling for the axes, so that in each color, the picture
  is approximately a standard (rate 1 per unit area) two-dimensional
  Poisson process.  The color scheme is intended to be
purple, green, orange.}\label{fig 6}
\end{figure}


\subsection{Toggling:  The Case $k=2$ and $m=1$.}\label{sect sub toggle}

We show what can happen when we toggle one bit of a logic $f$.  We
have two segments of length $t$, which share a vertex $\vs$.  Toggling
changes the value of $f$, only at $\vs$, and gives a new logic $\fs$.
Suppose that the segments under $f$ were red and yellow, and that
$\vs$ appears at position $i$ on the red segment, and position $j$ on
the yellow segment.  Overall, this repeat has spatial location
$(i,j)$, and color orange.  Exactly one such repeat was visualized in
Figures~\vref{fig 2} and \vref{fig 3};  it occurred with $(i,j)=(56,53)$.
The \emph{displacement} is $i-j$  
 --- we have a preferred sequence of
colors,  (derived from the rainbow ROY G. BIV)
where red comes before yellow --- hence 
the displacement is 3, rather than $-3$, in this example.

\begin{figure}
   \begin{center}
   \begin{tikzpicture}[scale=.10]
   \draw[fill=red,ultra thick,red] (0,40) -- (90,40);
   \draw[fill=yellow,ultra thick,yellow] (0,30) -- (90,30);
   \draw [fill=orange,orange] (56,40) circle [radius=1];
   \draw [fill=orange,orange] (53,30) circle [radius=1];

   \draw[fill=red,ultra thick,red] (0,10) -- (56,10);
   \draw[fill=yellow,ultra thick,yellow] (56,10) -- (93,10);
   \draw[fill=yellow,ultra thick,yellow] (0,0) -- (53,0);
   \draw[fill=red,ultra thick,red] (53,0) -- (87,0);
   \draw [fill=orange,orange] (56,10) circle [radius=1];
   \draw [fill=orange,orange] (53,0) circle [radius=1];
   
   \draw[dotted]  (90,25)--(90,45);
   \draw[dotted]  (93,15)--(93,5);
   \draw[dotted]  (87,5)--(87,-5);

   \draw (-10,37)  node[anchor= north west]{$f$};
   \draw (-10,7)  node[anchor= north west]{$\fs$};
   \end{tikzpicture}
   \end{center}   
\caption{Toggling.  An example with $t=90$, and displacement $d=3$.  The same
     repeat as shown in Figure~\vref{fig 3} with location (56,153), and shown
 by the orange dot in Figure~\vref{fig 4}.
When all ${k \choose 2}$ squares are superimposed, as in Figure~\vref{fig 6}, the spatial location becomes $(i,j)=(56,53)$.   Before
the
toggle, we have two segments of length $t=90$;  after the toggle, the
segments have length $t \pm d$, that is, 93 and 87.
}\label{fig 7}
\end{figure}
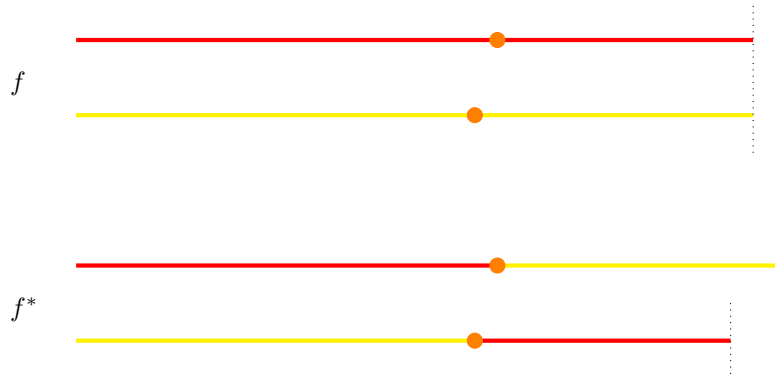
 

\subsection{Picking the ``Earliest'' Toggle with a Small Displacement}
\label{sect earliest}

Consider the case where we have $k=2$ segments, and
want to find a single vertex $\vs$ 
via a recipe which, when applied to the segments under the toggled
logic $\fs$, still picks out the same vertex.  A good recipe involves
naming a small bound $d$ on the absolute displacement $|i-j|$ (thus
staying close to the ``diagonal''),
and then picking the ``earliest'' pair $(i,j)$ that satisfies
the displacement bound. This was the key to overcoming the ``fallacy''
described in Footnote~\ref{footnote fallacy}.

The specific choice of how to define \emph{earliest} is somewhat
arbitrary;  we will take smallest $(i+j)$ as the first criterion for 
earliest, with ties to be broken
according to smallest value of $\max(i,j)$ --- given that $i+j =
i'+j'$, this is equivalent to taking smallest absolute displacement
for the tie-break criterion. For use in the case of $k$ colors and
  ${k \choose 2}$ color pairs $\alpha=(a,b)$, break further ties
  according
to $\min(a,b)$ and then $\max(a,b)$.

The logic $f$, with value at $\vs$ complemented, gives a new logic
$\fs=\toggle(f,\{\vs \})$, so that $\fs(\vs)=1-f(\vs)$, while
$\fs(v)=f(v) \ \ \forall v \ne \vs$.  It is \emph{possible} that
changing the logic bit at $\vs$, will cause an earlier pair to become
available as the location of a match between the two segments; so that
the recipe for picking the earliest small displacement match, applied
to $\fs$, picks out a \emph{different} vertex instead of $\vs$. In
this case, the word \emph{toggle} is very misleading: the overall
operation (find $\vs$, then complement the logic at that vertex) is
not an involution.  Our program is to specify a displacement bound $d$
that varies with $n$, in such a way that 1) with high probability, at
least one small displacement match can be found, and 2) with high
probability, the vertex for the earliest small displacement match is
the same in the logic $\fs =\toggle(f,\{\vs \})$  at the vertex selected for
$f$.  The example in Figure~\vref{fig 8}, viewed with any $d \ge 3$,
illustrates what might go wrong with respect to 2).

Recall, from Section \ref{sect coins}, that $t$ is the length of our
segments. 
To get high probability in 1),  a necessary and sufficient condition is that 
\begin{equation}\label{need 1} 
   td/N \to \infty.
\end{equation}
To get high probability in 2), a necessary and sufficient condition is
that 
\begin{equation}\label{need 2}
d^2/N \to 0.
\end{equation}
The argument that~\eqref{need 2} suffices is somewhat delicate, akin
to a stopping time argument; it is easier to prove --- see 
\eqref{overlapping squares}  --- that a
\emph{sufficient}
condition is that
\begin{equation}\label{need 3}
td^3/N^2 \to 0;
\end{equation}
and then it will be easy to arrange for situations corresponding to
pairs
$(t,d)$ satisfying both~\eqref{need 1} and~\eqref{need 3}.

\begin{figure}
   \begin{center}
   \begin{tikzpicture}[scale=.10]
   \draw[fill=red,ultra thick,red] (0,40) -- (90,40);
   \draw[fill=yellow,ultra thick,yellow] (0,30) -- (90,30);
   \draw [fill=orange,orange] (56,40) circle [radius=1];
   \draw [fill=orange,orange] (53,30) circle [radius=1];
   \draw [fill=red,red] (53,40) circle [radius=1];
   \draw [fill=red,red] (58,40) circle [radius=1];
  
   \draw (-10,37)  node[anchor= north west]{$f$};
   \draw (-10,7)  node[anchor= north west]{$\fs$};

   \draw[fill=red,ultra thick,red] (0,10) -- (56,10);
   \draw[fill=yellow,ultra thick,yellow] (56,10) -- (93,10);
   \draw[fill=yellow,ultra thick,yellow] (0,0) -- (53,0);
   \draw[fill=red,ultra thick,red] (53,0) -- (87,0);
   \draw [fill=orange,orange] (56,10) circle [radius=1];
   \draw [fill=orange,orange] (53,0) circle [radius=1];
   \draw [fill=red,red] (53,10) circle [radius=1];
   \draw [fill=red,red] (55,0) circle [radius=1];   
   \draw[dotted]  (90,25)--(90,45);
   \draw[dotted]  (93,15)--(93,5);
   \draw[dotted]  (87,5)--(87,-5);

   \end{tikzpicture}
   \end{center}   
\caption{Toggling.  This is a continuation of the example in
  Figure~\vref{fig 7}, with one repeat with location (56,153), shown
  by the orange dot in Figure~\vref{fig 4}.  When all ${k \choose 2}$
  squares are superimposed, as in Figure~\vref{fig 6}, the spatial
  location becomes $(i,j)=(56,53)$.  Now suppose there were an
  additional repeat, (which would have been shown by a red dot at
  (53,58) in Figure~\vref{fig 4},) shown here in Figure~\vref{fig 8} by
  the pair of red dots for $f$.  After the toggle at the orange
  vertex, vertex 53, along the segment that starts red and finishes
  yellow, is the same as vertex 55, along the segment that start
  yellow and finishes red.  So, in the logic $\fs$, we have two
  matches between the two segments: the original, at (56,53), shown by
  the orange dots, and a new one, at (53,55), shown by the red dots.
}\label{fig 8}
\end{figure}
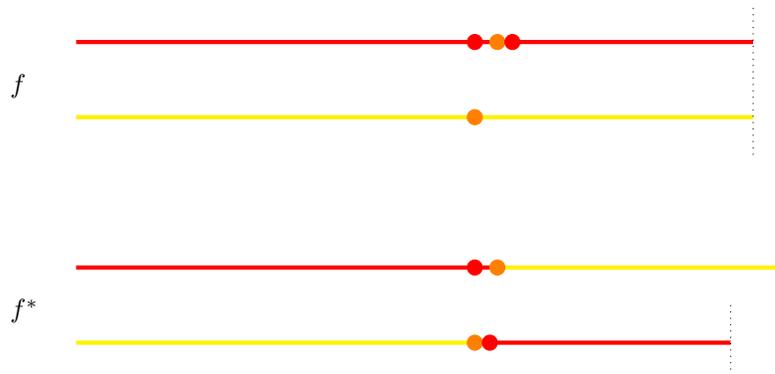
 

\subsection{Displacements Caused by Toggles}\label{sect displace}

Suppose we have $k=3$ colors, as shown in Figure~\vref{fig 9}.
There are three segments of length $t=90$, with respect to $f$.
The segment  with respect to $f$, starting with $e_{1}$, colored
red,  has $\vs_1$ in position 6 and $\vs_2$ in position 35 --- so
the red segment, of length 90, is divided into an initial red path of length 6, followed by a red path of
length 29, followed by a red path of length 55.

The $f$ segment starting with $e_{2}$, colored yellow, has $\vs_1$
in position 3, and $\vs_3$ in position 75,  hence it is divided
into  yellow paths of lengths 3, 72, 15, in that order.

The $f$ segment starting with $e_{3}$, colored blue, has $\vs_2$
in position 37, and $\vs_3$ in position 71,  hence it is divided
into  blue  paths of lengths 37, 34, 19, in that order.

Next, consider $\fs := f$, toggled at $\vs_1$.  Its segment starting
from $e_{1}$ has length 6 red followed by length (72+15)=87, for a
total
length of 93.   Its segment starting from $e_{2}$ has length 3
yellow, followed by length (29+55)=84, for a total length of 87.   The
$\fs$ segment starting from $e_{3}$ is still length 90, all blue.
More importantly, $\vs_2$ has moved from position 35 on $\seg(f, 
e_1)$ to position 32 on $\seg(f,e_2)$,  and  
$\vs_3$ has moved from position 75 on $\seg(f, 
e_2)$ to position 78 on $\seg(f,e_1)$, so these have new positions
under $\fs$, \textit{i.e.}, have been displaced.

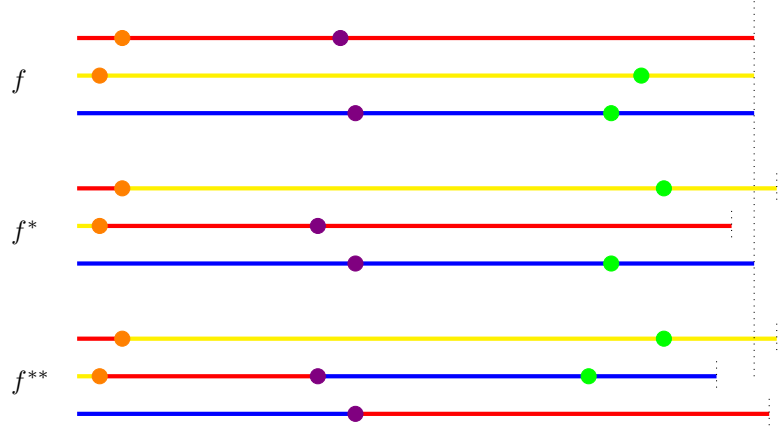
\begin{figure}
   \begin{center}
   \begin{tikzpicture}[scale=.10]

   \draw[fill=red,ultra thick,red] (0,40) -- (90,40);
   \draw[fill=yellow,ultra thick,yellow] (0,35) -- (90,35);
   \draw[fill=blue,ultra thick,blue] (0,30) -- (90,30);

   \draw [fill=orange,orange] (6,40) circle [radius=1];
   \draw [fill=orange,orange] (3,35) circle [radius=1];
   \draw [fill=violet,violet] (35,40) circle [radius=1];
   \draw [fill=violet,violet] (37,30) circle [radius=1];
   \draw [fill=green,green] (75,35) circle [radius=1];
   \draw [fill=green,green] (71,30) circle [radius=1];

   \draw (-10,37)  node[anchor= north west]{$f$};
   \draw (-10,17)  node[anchor= north west]{$\fs$};
   \draw (-10,-3)  node[anchor= north west]{$\fss$};
   \draw[fill=red,ultra thick,red] (0,20) -- (6,20);
   \draw[fill=yellow,ultra thick,yellow] (6,20) -- (93,20);
   \draw[fill=yellow,ultra thick,yellow] (0,15) -- (3,15);
   \draw[fill=red,ultra thick,red] (3,15) -- (87,15);
   \draw[fill=blue,ultra thick,blue] (0,10) -- (90,10);
 
   \draw [fill=orange,orange] (6,20) circle [radius=1];
   \draw [fill=orange,orange] (3,15) circle [radius=1];
   \draw [fill=violet,violet] (32,15) circle [radius=1];
   \draw [fill=violet,violet] (37,10) circle [radius=1];
   \draw [fill=green,green] (78,20) circle [radius=1];
   \draw [fill=green,green] (71,10) circle [radius=1];

   \draw[fill=red,ultra thick,red] (0,0) -- (6,0);
   \draw[fill=yellow,ultra thick,yellow] (6,0) -- (93,0);
   \draw [fill=orange,orange] (6,0) circle [radius=1];
   \draw [fill=green,green] (78,00) circle [radius=1];
   \draw[fill=yellow,ultra thick,yellow] (0,-5) -- (3,-5);
   \draw[fill=red,ultra thick,red] (3,-5) -- (32,-5);
   \draw[fill=blue,ultra thick,blue] (32,-5) -- (85,-5);
   \draw [fill=orange,orange] (3,-5) circle [radius=1];
   \draw [fill=violet,violet] (32,-5) circle [radius=1];
   \draw [fill=green,green] (68,-5) circle [radius=1];
   \draw[fill=blue,ultra thick,blue] (0,-10) -- (37,-10);
   \draw[fill=red,ultra thick,red] (37,-10) -- (92,-10);
   \draw [fill=violet,violet] (37,-10) circle [radius=1];

   \draw[dotted]  (90,-5)--(90,45);
   \draw[dotted]  (93,22)--(93,18);
   \draw[dotted]  (87,17)--(87,13);
   \draw[dotted]  (93,2)--(93,-2);
   \draw[dotted]  (85,-3)--(85,-7);
   \draw[dotted]  (92,-8)--(92,-12);
 
   \end{tikzpicture}
   \end{center}   
\caption{With starting edges $e_1,e_2,e_3$,  three segments under the
  logic $f$ are shown in the top part of the display;  
the red and yellow segments
  share a vertex $\vs_1$, colored orange, early on,  
the red and blue segments share a vertex $\vs_2$, colored purple, at
  a intermediate time, and 
the yellow and blue segments share a vertex $\vs_3$, colored green, at a late
  time.   We take  $\fs=\toggle(f,\{\vs_1 \})$ and 
 $\fss=\toggle(f,\{\vs_1,\vs_2 \})$ to be the logics formed by
  toggling at $\vs_1$, and at both $\vs_1$ and $\vs_2$.   The middle
  part of the display shows the three segments under $\fs$, and the
  bottom part of the display shows the three segments under $\fss$.
}\label{fig 9}
\end{figure}
 

Now consider the full effect of changing from $f$ to $\fs$, by
toggling the logic at the bit $\vs_1$ which appeared at positions
$(i,j)=(i,i-d) = (6,3)$, with $d=3$, for the red and yellow segments:
every red vertex later than 6 gets displaced by $-d$, and every yellow vertex 
later than 3 gets displaced by $+d$.  If a match 
occurs at $(I,J)$ in the $f$ segments, and the colors involved are
red, and some color, call it $a$, with $a$ not equal to yellow, then:
\begin{itemize}

\item Case 1.  Color $a$ comes after red, in the list of $k$ colors:
  the ordered color pair is (red,$a)$.
  The index $I>i$ belongs to a red vertex in position $I$ under
  the logic $f$, and this vertex has position $I-d$ under the logic $\fs$.
  So the point at $(I,J)$ \emph{moves} to position $(I-d,J)$.\footnote{More
    formally,
the point at $(I,J)$, labeled by the pair of colors ($a$,red), in the
colored-spatial process of indicators of matches between segments
under $f$, \emph{corresponds} to a point at $(I-d,J)$ in the colored-spatial
process for $\fs$.}

\item Case 2.  Color $a$ comes before  red, in the list of $k$ colors;
the color pair is $(a,$red).
  The index $J>i$ belongs to a red vertex, in position $J$ under
  the logic $f$, but in position $J-d$ under the logic $\fs$.
  So the point at $(I,J)$ moves to position $(I,J-d)$.

\end{itemize}

If there is an orange match at $(I,J)$ for the $f$ segments, with
$I>i$ and $J>j$, this match will move to $(I-d,J+d)$.

Similarly, a match between yellow, and some $a$ not equal to red,
occurring at $(I,J)$ under the logic $f$, moves to $(I+3,J)$ or
$(I,J+3)$ under the logic $\fs$, according to whether $a$ comes
after or before yellow, in the list of all $k$ colors.

This effect can be seen in Figure~\vref{fig 9}:  the orange dot is at
(6,3) with displacement $d=3$, the purple dot occurs
at (35,37) under $f$, but at (32,37) under $\fs$ and $\fss$.

More cases can be seen in Figure~\vref{fig 10}.

\begin{figure}
   \begin{center}
   \begin{tikzpicture}[scale=.10]

   \draw (-10,-10)  node[anchor= north west]{$f$};
   \draw (-10,-25)  node[anchor= north west]{$\fs$};
   \draw[fill=red,ultra thick,red] (0,-10) -- (90,-10);
   \draw[fill=blue,ultra thick,blue] (0,-15) -- (90,-15);
   \draw [fill=violet,violet] (35,-10) circle [radius=1];
   \draw [fill=violet,violet] (40,-15) circle [radius=1];

   \draw[fill=red,ultra thick,red] (0,-25) -- (35,-25);
   \draw[fill=blue,ultra thick,blue] (35,-25) -- (85,-25);
   \draw[fill=blue,ultra thick,blue] (0,-30) -- (40,-30);
   \draw[fill=red,ultra thick,red] (40,-30) -- (95,-30);
   \draw [fill=violet,violet] (35,-25) circle [radius=1];
   \draw [fill=violet,violet] (40,-30) circle [radius=1];

   \draw [fill=violet,thick] (35,40) circle [radius=1];

   \draw[dotted]  (0,0)--(90,0);
   \draw[dotted]  (0,0)--(0,90);  
   \draw[dotted]  (90,0)--(90,90);
   \draw[dotted]  (0,90)--(90,90);
   \draw[dotted]  (0,0)--(90,90); 
   \draw[dotted]  (35,0)--(35,90); 
   \draw[dotted]  (0,40)--(90,40);
   \draw [fill=violet,violet] (5,13) circle [radius=1];
   \draw [fill=orange,orange] (15,9) circle [radius=1];

   \draw [fill=orange,orange] (50,52) circle [radius=1];
   \draw [orange] (55,52) circle [radius=1];

   \draw [fill=orange,orange] (40,25) circle [radius=1];
   \draw [orange] (45,25) circle [radius=1];

   \draw [fill=green,green] (20,60) circle [radius=1];
   \draw [green] (20,55) circle [radius=1];

   \draw [fill=violet,violet] (70,82) circle [radius=1];
   \draw [violet] (75,77) circle [radius=1];

   \end{tikzpicture}
   \end{center}   
\caption{Displacements caused by a single toggle.  An example with
  $t=90$, and three colors, red, yellow, blue.  Say the toggle is at
  $\vs_2$ occurring at (red,blue) time $(35,40)$, similar to the purple
  vertex at ($35,37)$ in Figure~\vref{fig 9}, but with the displacement
  changed from -2 to -5, for the sake of being easier to see in the
  two-dimensional picture.  We have thrown in several more matches
  between two different colors, at various earlier and later times, to
  show the resulting two-dimensional displacements.  Red vertices at
  times greater than 35 have their time increased by 5, and blue
  vertices at times greater than 40 have their time decreased by 5.
  Two-dimensional match locations are indicated by a solid circle for
  the logic $f$, and an open circle for the logic $\fs$.  }\label{fig
  10}
\end{figure}
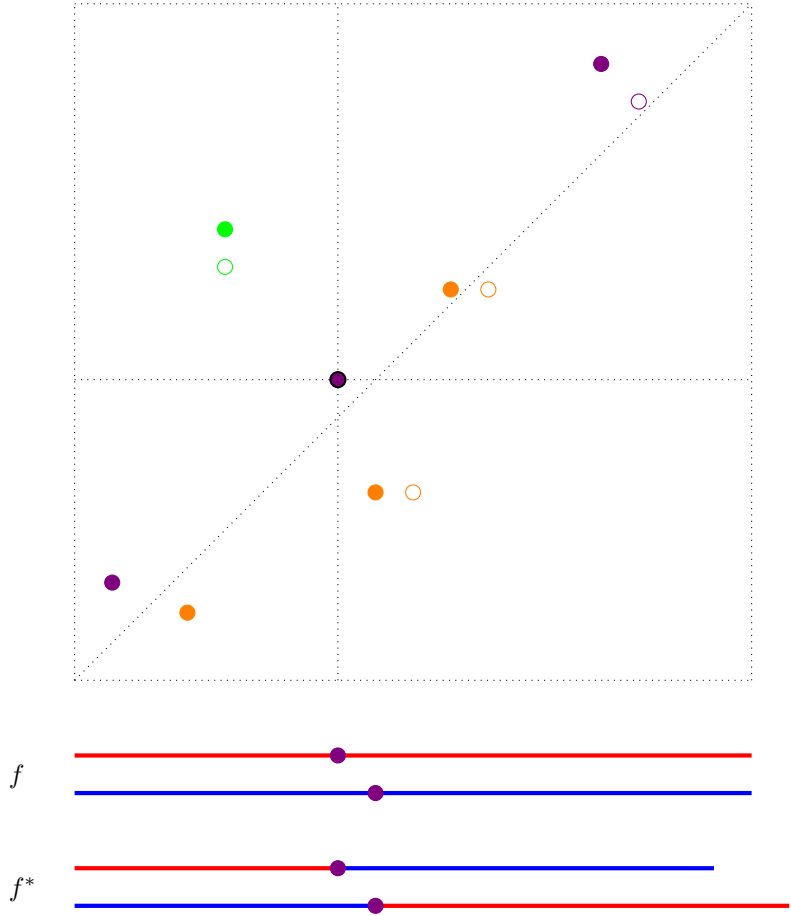
 

\subsection{The Natural Scale:  by $1/\sqrt{N}$ for length, by $1/N$
  for area}\label{sect scale}

One gets an intuitive grasp of the \emph{process of spatial locations}
of places $(i,j)$ where two segments of different colors share a
vertex, by looking at a picture such as that in Figure~\vref{fig 6}
--- even though the axes are unlabeled.

One view would be that the square is $t$ by $t$,  with $n=34, N=2^n,t=N^{.6}
$, \textit{i.e.}, about 1.3 million by 1.3 million.  The other natural view is
that the square is about $t/\sqrt{N}$ by $t/\sqrt{N}$, \textit{i.e.}, about
10.556 by 10.556, with area 111.43.

The latter point of view is natural, since at each $(i,j)$, for
each color pair
$(a,b)$, $1 \le a < b \le k$, with $\doteq$ to allow a small discrepancy for the failure of
the good event, $\p($an \emph{arrival}\footnote{This jargon comes from
  queuing theory and Poisson arrival processes; we say there is an
  arrival at $(i,j)$ if the indicator indexed by $(i,j)$ takes the
  value 1, here indicating that there is an $(n-1)$-tuple repeat.} at $(i,j)$ in those colors) 
$  := \p(v_{a,i}=v_{b,j}$ and $v_{a,i-1} \ne v_{b,j-1})$
$ \doteq \p($there is a
leftmost $(n-1)$-tuple repeat at a specific location\footnote{The
  precise location doesn't matter, but, using Section \ref{sect cutting}, 
  the location is $(i_0,j_0)$ where
$i_0 = i+(n+t)(a-1)$ and $j_0=j+(n+t)(b-1)$.
}
 in the coin tossing sequence)
$= 1/N$.  Hence, scaling length   by $1/\sqrt{N}$, so that area is
scaled by $1/N$, leads to
$$
   \text{ the expected number of arrivals per unit area } =1.
$$

The picture in Figure~\vref{fig 6}, viewed as occurring on a 10.556 by
10.556 square, closely resembles a (standard, rate $1 \, dy \, dx$)
two-dimensional Poisson process, in each secondary color pair.  And
overall, ignoring color, the picture resembles the rate ${k \choose 2}
\, dy \, dx$ Poisson process on the $t/ \sqrt{N}$ by $t/ \sqrt{N}$
square.

There are additional requirements for the Poisson process, beyond
  having
intensity  $1 \, dy \, dx$.  Namely, probabilistic independence for
the counts in disjoint regions.  We do have a good Poisson process
approximation, for a combination of two reasons.
First, the good event $G=\Gkt$ from Theorem \ref{kt theorem} gives a
high-probability coupling (since $t=N^{.6}$ entails $t^3n^3/N^2 \to 0$) 
between coin tossing and the $k$ de Bruijn segments of length $t$.
Second, the Chen-Stein method, Theorem 3 of \cite{AGG89}, gives a 
total-variation distance
upper bound (tending to zero since $t=N^{.6}$ entails $t^3n/N^2 \to 0$) 
between the process of indicators of leftmost
$(n-1)$-tuple repeats
for coin tossing, and a process with the same intensity, but mutually
independent coordinates.

We get our intuition from the Poisson process.  But for our proofs, we
will
work directly with the discrete, dependent processes.

\subsection{Controlled Regions for $m$ Successive Potential Toggle
  Vertices}

\subsubsection{Quick Motivation for the Geometric Progression}
\label{sect motivate regions}

We will construct \emph{choice} functions in \eqref{def choice}, based on 
\emph{regions}, defined in \eqref{def region}, which in turn are
based on a geometric progression
in~\eqref{geometric}.  Here we give some motivation for this elaborate construction.

If we search for a single toggle point, in a thin and long rectangle
along the diagonal, $\{ (i,j): |i-j| \le d, 0\le i,j \le t \}$, then,
in the natural scale of Section~\ref{sect scale}, (and ignoring
factors of $\sqrt{2}$ related to the 45 degree rotation, and of 2 for
$\pm d$), the rectangle is $d_1 = d/\sqrt{N}$ by $w_1= \ t/\sqrt{N}$.
Condition~\eqref{need 1} can be interpreted as meaning that the (natural scale) area, $d_1w_1$,
tends to infinity --- so that with high probability, matches can be
found in this rectangle, and condition~\eqref{need 2} can similarly be
interpreted as meaning that $d_1 \to
0$, so that no matches will be found in the two-dimensional set, of
area on the order of $d_1^2$, of points within $\ell_\infty$ distance
$d_1$ of the chosen location $(i,j)$.

Now in choosing $m$ toggle points, displacements caused by earlier
toggles might change the search result, and we wish to make this
unlikely.
In more detail: as seen in Section~\ref{sect displace},
toggling
a logic $f$ at a vertex $\vs$ which appears on two different colors,
at times $i,j$ with $|i-j| \le d$ \emph{causes} displacements in the
time
indices of vertices occurring later on those segments, by amounts up to
$d$.
Our  $m$ potential toggle points,
$\vs_1,\ldots,\vs_m$, are controlled so that on any segment, $\vs_\ell$
is
preceded by toggle points from among the $\vs_1,\ldots,\vs_{\ell-1}$.
If the displacement caused by toggling at $\vs_i$ is at most $d_i$,
then in choosing $\vs_\ell$,  the accumulated 
displacements from previous toggles is  at most
$d_1+\cdots+d_{\ell-1}$.
By taking the $d_i$ in geometric progression, with large ratio $r^2$,
this accumulated displacement in the search for $\vs_\ell$ is at most
order of $d_{\ell-1}$.
The rectangle where we search for $\vs_\ell$ is thin and long, 
$d_\ell$ by $w_\ell =
r/d_\ell$;
the length of its boundary is order of $w_\ell$, so the area involved 
in points at a distance at most $d_{\ell-1} = d_\ell/r^2$ from the boundary is
order of $1/r =o(1)$.  Hence with high probability,
displaced indices have no effect.

\subsubsection{The Search Regions}

We divide the time interval $[0,t]$ into $m$ equal length pieces.
On the earliest piece, with times in [0,t/m], we demand that we can find a match $(i,j)$ with
$|i-j|$ very very very small, but no upper bound on $\max(i,j)$
other than $\max(i,j) < t/m$. In the \emph{natural scale} of Section
\ref{sect scale} we are searching for matches in a very
very very thin and very very long rectangle surrounding the diagonal
line $i=j$; this rectangle has a large area.  On the
second piece, with times in $[t/m,2t/m]$, 
we relax the notion of thin, expanding by a large
factor $r^2$, relax the notion of  long, dividing by the factor $r^2$,
thus keeping the area constant.
  We continue this
pair of geometric
progressions, so the $m$th region is a thin long rectangle --- but
still with the same area.   

Here is a concrete way to accomplish the above, together with $t^3n^3/N^2 \to 0$ and with $k$ fixed.  Let

$$
    t := m \, N^{.6}, a := N^{.1}, \text{ so that } 
t/m = a \sqrt{N}.
$$
The last condition should be understood as ``in the natural scale
from Section~\ref{sect scale}, the $t$ by $t$ rectangle is $ma$ by
$ma$, and length $t/m$ for the discrete $i$ and $j$  corresponds to
length $a$''.
Let 
$$
   r := a^{1/(2m+1)},   \text{ so that }  r^{2m+1} = a, 
$$
and, ignoring the factors of $\sqrt{2}$ involved in the 45 degree
rotation,
 take the thin long rectangles to have shapes
\begin{eqnarray}
  d_1 = \frac{1}{r^{2m}} & \text{ by } & w_1 = r^{2m+1} = (t/m)/\sqrt{N} \nonumber\\
  d_2 =   \frac{1}{r^{2m-2}} & \text{ by } & w_2 = r^{2m-1}  \label{geometric}  \\
   & \vdots & \nonumber  \\
d_m = \frac{1}{r^2} & \text{ by } & w_m = r^{3}  \nonumber 
\end{eqnarray}

Indexing by $\ell=1$ to $m$, the $\ell$th rectangle is $d_\ell :=
r^{2\ell-2m-2}$ by $w_\ell := r^{2m-2\ell+3}$ on the 
natural scale. 
Directly in terms of the discrete $i$ and $j$,
we define
\begin{eqnarray}
\region_\ell =\left\{ (i,j):  \frac{|i-j|}{\sqrt{N}} < r^{2\ell-2m-2} \right. 
  & \text{ and }  &  \label{def region}  \\
\left. \frac{(\ell-1)t}{m} \le \min(i,j) \le
\max(i,j) \le \frac{(\ell-1)t}{m} + \frac{t/m}{r^{2(\ell-1)}} 
\right\},  & &   \nonumber 
\end{eqnarray} 
so one checks that  1) as $\ell$ increases by 1, the thinness
constraint
relaxes by a factor of $r^2$, while the width constraint becomes more
severe by a
factor of $r^2$, so the area stays constant,  2) the first region, with $\ell=1$, allows $i,j \in
[0,t/m]$, and 3) the last region, with $\ell=m$, has $|i-j|/\sqrt{N}
\le 1/r^2 = o(1)$ as $n \to \infty$. 

Consider the possibility discussed in Section~\ref{sect earliest},
 where a toggle at a vertex appearing in
two differently colored segment enables a match within a single 
segment to become, after the toggle, an earlier match between two
different segments.  For each $\ell=1$ to $m$, with the $(t,d)$ in~\eqref{need 3} given by
$t=w_\ell \sqrt{N}, d= d_\ell \sqrt{N}$, the condition in
\eqref{need 3} is indeed satisfied by our specific choice in~\eqref{geometric}.
On the natural scale, and ignoring rotation, we are searching for a
match
in a $\delta = d_\ell$ by $W=w_\ell$ rectangle, thin and long, with
$\delta \to 0$ and area $\delta W \to \infty$.  The condition 
\eqref{need 3}, on the natural scale, means that $\delta^3 W \to 0$.   It
implies that, with high probability, we do not find a match between
two differently colored segments (at
$(i,j)$ in the rectangle, with $|i-j|/\sqrt{N}<\delta$,) and
simultaneously a nearby match within a single segment.  Here, nearby
means
with both indices within distance $\delta \sqrt{N}$ from $i$ or $j$.
Now, the $\delta$ by $W$ rectangle can be covered by $W/\delta$
squares, each square of size $4 \delta$ by $4 \delta$, and with each
successive square being a translate, by $\delta$, of the previous
square.
Ignoring 
constant 
factors,\footnote{such as ${k \choose 2}+k$ --- for the intensity of
arrivals
in the superimposed process marking matches between two different
colors or both within the same color, and 16 --- 
since a $4 \delta$ by $4 \delta$ square has area  $16 \delta^2$}
 the expected number of arrivals in one square is order
of $\delta^2$, and
the chance of two or more arrivals in that one square is order of
$\delta^4.$
Thus the expected number of squares with two or more arrivals is order
of 
\begin{equation}\label{overlapping squares}
W/\delta \times \delta^4 \ = \delta^3 W \ \ \ \to 0.
\end{equation}
 
\subsection{Definition of the Choice Functions}

Write $V = \mathbb{F}_2^{n-1}$ for the set of vertices in $D_{n-1}$,
and write ``null'' for a special value, not in $V$, used to encode
``undefined''.  Recall that we write $\be = (e_1,\ldots,e_k)$ for the
starting $n$-tuples for $k$ segments, and $S_{n,k} = \{(f,\be) \}$ for
the space in which we make a uniform choice of logic and starting
edges.  Also recall our notation~\eqref{segment vertex} for vertices
along
the $k$ segments.
Note that we have both $k$
segments and $k$ colors;  these are different concepts, and
ultimately, \emph{colors} will be labeled according to the segment
labels under $f$  --- but on the
soon to be defined ``happy'' event $H$, finding $\vs_i$
 on \emph{two different} segments of $f$ will be equivalent to finding $\vs_i$
on  \emph{two different} colors. 
To keep track of the colors, let
\begin{equation}\label{def alpha}
  \cA := \{\alpha = (a,b): 1 \le a < b  \le k  \}
\end{equation}

For $\ell=1$ to $m$, we define
\begin{equation}\label{def candidates}
\candidates_\ell:   S_{n,k} \to [0,t]^2  
 \times \cA  
\end{equation}
$$
 \candidates_\ell(f,\be) = \{ (i,j,a,b):  (i,j) \in \region_\ell \text{
   and }  v_{a,i} = v_{b,j} \}
$$
where  $\region_\ell$ is defined by~\eqref{def region}.

For $\ell=1$ to $m$, we define
\begin{equation}\label{def choice}
\choice_\ell:   S_{n,k} \to V \cup \{ \nil \},
\ \ \choice_\ell(f,\be) = \vs_\ell  \text{ or else } \nil
\end{equation}
where the value is $\nil$ if the set of candidates is empty, and
otherwise,
picking the first $(i,j,a,b)$ in  $\candidates_\ell(f,\be)$,
$\vs_\ell$ is the vertex with   $\vs_\ell =v_{a,i} = v_{b,j}$.
To be very careful, the order for \emph{first} is the lex-first order
on
$(i+j, \max(i,j), a,b)$.

\subsection{The Happy Event $H = H(k,m,n)$}
\label{sect happy event}

We now describe a subset of $S_{n,k}$, and refer to this subset as the
\emph{happy} event~$H$.  One requirement for $(f,\be) \in H$ is that,
for
 $\ell=1$ to $m$, each of the values $\choice_\ell(f,\be) \ne \nil$.
Starting with such an $(f,\be)$, the choice functions pick out a set
of $m$ distinct vertices;  call them  $\vs_1,\ldots, \vs_m$, and name
the set, $\vtwiddle = \{ \vs_1,\ldots, \vs_m \}$ --- we will use this
notation in~\eqref{def H} below.

Given a set of vertices, $U \subset V$, we denote \emph{the logic
  $f$ toggled at the vertices in $U$} as $ \toggle(f,U)$,
defined by
\begin{equation}\label{def toggle}
  \toggle(f,U) := \fs,  \text{ where } \fs(v) = 
   \left\{   \begin{array}{ll}
                 1-f(v)  & \mbox{if $v \in U$} \\
                 f(v) & \mbox{if $v \in V \setminus U$}
             \end{array}
    \right.   .
\end{equation}

We define $H$ as follows:
\begin{equation}\label{def H}
H = \{ (f,\be):  \forall \ U \subset
\vtwiddle, \text{ with } \fs = \toggle(f,U), \vs_\ell = \choice_\ell(\fs,\be)
\end{equation}
$$
 \text{ \emph{and} the segments } \seg(f,e_i,t) \text{ collectively have }
k(t+1) \text{ distinct edges} \}.
$$
Informally, $(f,\be)$ is in the happy event iff 
the $k$ segments involve no $n$-repeats, and the choice recipes     
find $m$ potential toggle vertices, and all $2^m$ \emph{cousins}
$\fs$, 
formed by
toggling at a subset of those vertices, give rise to the 
same $\vs_1,\ldots,\vs_m$.

The definition above creates an equivalence relation on $H$, in which
all classes have size $2^m$, and all $(\fs,\be) \in [(f,\be)]$ share
the same sequence $\vs_1,\ldots,\vs_m$.   Using the calculations given
in Section~\ref{sect motivate regions} one may show that
for fixed $k,m$,
$|H|/|S_{n,k}| \to 1$; that it, that $\p(H) \to 1$ as $n \to \infty$.


\subsection{Definition and Likelihood of 
an $\varepsilon$-good Schedule}\label{sect schedule}

Given $k$, view $\cA$, defined by~\eqref{def alpha}
as an \emph{alphabet} of size
$$   
   K := {k \choose 2}.
$$

A schedule of length $m$ is a word $\alpha_1 \alpha_2 \cdots \alpha_m
\in \cA^m$.   Given a schedule of length $m$, and $m$ coin tosses
$D_1,\ldots,D_m$,
for $i=1$ to $m$ define permutations in $\cS_k$ by
$$
   \tau_i =  \left\{   \begin{array}{ll}
             \text{the transposition } (a b)  & \mbox{if
               $\alpha_i=(a,b)$ and $D_i=\text{ heads }$} \\
                 \text{ the identity } & \mbox{if $D_i=\text{ tails }$}
             \end{array}
    \right. ,   
$$
and let $\tau = \tau(\alpha_1 \alpha_2 \cdots \alpha_m,D_1,\ldots,D_m)$
be the product, with $\tau_1$ applied first,
\begin{equation}\label{on schedule}
   \tau = \tau_m \circ \cdots \circ \tau_2 \circ \tau_1 \in \cS_k.
\end{equation}

Write $\sigma$ for an arbitrary permutation in $\cS_k$, and let
$$
  p_\sigma = p_\sigma(\alpha_1 \alpha_2 \cdots \alpha_m) = \p(\tau =
  \sigma | \alpha_1 \alpha_2 \cdots \alpha_m )
$$
be the conditional probability of getting $\sigma$ for the value of
$\tau$,
given the schedule $\alpha_1 \alpha_2 \cdots \alpha_m$ --- these are values of the form $z/2^m$ with $z$ in $Z$.   The total variation distance to the uniform
distribution on $\cS_k$
is
$$
 \distance(\alpha_1 \alpha_2 \cdots \alpha_m)=\dtv(\tau,\text{uniform})
= \frac{1}{2} \sum_\sigma \left| p_\sigma - \frac{1}{k!} \right|.
$$
Given $\varepsilon>0$, a schedule  $\alpha_1 \alpha_2 \cdots \alpha_m$
is
$\varepsilon$-good if $  \distance(\alpha_1 \alpha_2 \cdots \alpha_m)
< \varepsilon$. 

\begin{lemma}\label{schedule lemma}
Given $k$, and $\varepsilon >0$, there exists $m$ such that, for a
random schedule of length $m$, with all ${k \choose 2}^m$ equally likely,
\begin{equation}\label{shuffle goal}
\p( \alpha_1 \alpha_2 \cdots \alpha_m \text{ is } \varepsilon
\text{-good})
    > 1-\varepsilon.
\end{equation}
\end{lemma}
\begin{proof}
There is a well-known bijection  
between $\cS_k$ and the set
$C_k := [1] \times [2] \times \cdots \times [k]$:  given 
$c=(c_1,c_2,\ldots,c_k)$ with $1 \le c_i \le i$, take
\begin{equation}\label{shuffle}
  \sigma = (2 \ c_2) \circ \cdots \circ (k-1 \ c_{k-1}) \circ (k \ c_k),
\end{equation}
where $(a \ b)$ denotes the transposition $(a \ b) \in \cS_k$ if $a \ne b$,
and the identity map otherwise.  (The corresponding algorithm, to
generate uniformly distributed random permutations, is known as the
``Fisher-Yates shuffle'' or ``Knuth shuffle''.)

Now consider the particular word $w$ of length $K$ over the alphabet $\cA$
defined in~\eqref{def alpha}, given by 
$$
   w= (1 \ 2) (1 \ 3) (2 \ 3) \cdots (k-2 \ k) (k-1 \ k).
$$
If we had $m=K$ and the schedule is $\alpha_1 \alpha_2 \cdots
\alpha_m=w$,
then $\distance(\alpha_1 \alpha_2 \cdots \alpha_m) \le 1-2^{-K}$,
because for each $\sigma$ in~\eqref{shuffle}, one
assignment of the coin values $(D_1,\ldots,D_m)$ 
yields $\tau=\sigma$, via the coins for the genuine transpositions
among the $(i \ c_i)$ on the right side of~\eqref{shuffle} being
heads, and all others coins being tails.  
When the word $w$ appears  
$\ell$ times inside a long
word $\alpha_1 \alpha_2 \cdots \alpha_m$,
we have, using a standard result,
$$
\distance(\alpha_1 \alpha_2 \cdots \alpha_m) \le \left(
1-2^{-K}\right)^\ell.
$$
For historical interest, we note that similar results are in \cite[Thm. 1, p. 23]{DiaconisGroup}; see also \cite{Diaconis}.
In a very long random
word $\alpha_1 \alpha_2 \cdots \alpha_m$, the number of occurrences of
$w$ is random, with mean and variance roughly $m \ K^{-K}$,  so a
sufficiently large $m$ guarantees that $\ell$ is sufficiently large,
with high probability.
\end{proof}

\subsection{Relativized Permutations}
\label{sect tangle}
We will define ``$\pi_f$
relativized to $e_1,\ldots,e_k$'' to be a specific permutation in $\cS_1 \cup
\cdots \cS_{k-1} \cup \cS_k$, where $\cS_j$ denotes the set of all
permutations on $\{1,2,\ldots,j\}$.  
For use in
 Lemma \ref{permutation lemma}, we need to allow for the possibility
 that $e_1,\ldots,e_k$ are not $k$ distinct $n$-tuples.

\begin{definition}\label{relativisation}
Let  $\pi$ be a permutation on a finite set $S$, and 
let $\be = (e_1,\ldots,e_k) \in S^k$.   
 In case $e_1,\ldots,e_k$ are all
distinct,
write the full cycle notation for $\pi$, erase all symbols not in
$\{e_1,\ldots,e_k\}$, and then relabel $e_1,\ldots,e_k$ as
$1,\ldots,k$. This yields the cycle notation for a permutation 
$\sigma = \sigma(\pi,\be) \in \cS_k$, and we call $\sigma$ ``$\pi$
relativized to $\be$''.  In case $j := | \{e_1,\ldots,e_k\}| < k$, edit
the list $(e_1,\ldots,e_k)$ by deleting repeats, from left to right,
 to get a new list
$\be' =(e_1',\ldots,e_j') \in S^j$, with no repeats.  Now we take   ``$\pi$
relativized to $\be$'' to be $ \sigma(\pi,\be') \in \cS_j$.
\end{definition}

On the happy event $H$ from~\eqref{def H}, consider an equivalence
class $[(f,\be)]$.  We want to name a canonical choice of class
leader, and since all $2^m$ elements $(\fs,\be)$ in the class share the same
$\vs_1,\ldots,\vs_m$, and differ only in the values of the $\fs$ at
those vertices, the natural choice of leader is $(f_0,\be)$ where
$$ 
   f_0(\vs_1) = \cdots = f_0(\vs_m) = 0.
$$
Finally, we can say what \emph{colors} are: for $a=1$ to $k$,
 the vertices along 
$\seg(f_0,e_a,t)$ have color $a$.
Among the various $(\fs,\be)$ in the equivalence class $[(f_0,\be)]$,
except for the case $\fs=f_0$, at least some of the $k$ segments
start with one color and end with another.

The schedule corresponding to the equivalence class $[(f,\be)]$
is the word $\alpha_1 \alpha_2 \cdots \alpha_m$ where $\alpha_i =
(a_i,b_i)$ where $1 \le a_i < b_i \le k$ and $\vs_i$ appears on colors
$a_i$ and $b_i$, that is, $\vs_i$ is a vertex of both
$\seg(f_0,e_{a_i},t)$ and $\seg(f_0,e_{b_i},t)$.  We visualize\footnote{This is a
 only a visualization, and not a technical definition.  Imagine $k$
 strands of (directed) yarn, of different colors.  They are all 
tangled up, but the  start and end of each strand protrude from the
tangle, 
so one
 has $2k$ protruding ends (one male, one female, in each color).
One only knows that inside the tangle, there are $m$ instances of two
 different colored yarns being cut, and at each of these $m$, both
 strands may be spliced back together in their original (no color
 change)
form, or else they may be cross-joined.}
$f(\vs_i)=1$ as meaning that the strands of colors $a_i$ and $b_i$ are
cut (at $\vs_i$) and glued together to create a color jump, as in
Figures~\vref{fig 8} and \vref{fig 9}.

For $a=1$ to $k$, write $\ee_a := $ the final edge $e_{a,t}$ of $\seg(f_0,e_a,t)$, so
that, under the logic $f_0$, $\seg(f_0,e_a,t)$ is a directed path (in color $a$) from its \emph{female} end $e_a$ \emph{to} its \emph{male}
end $\ee_a$.  Note that being in $H$ implies that the starting edges 
$e_1,\ldots,e_k$ are distinct, and the final edges $\ee_1,\ldots,\ee_k$ are distinct.

It is clear --- from the relative timing of the appearances of the
$\vs_1,\ldots,\vs_m$ along the segments $\seg(f_0,e_a,t)$ --- that
under the logic $\fs$, $\seg(\fs,e_a,t)$ is a directed
  path from its female end $e_a$ \emph{to} its male
end   $e_{g(a)}'$,  
 where $g \equiv g(\fs)$ is the permutation in $\cS_k$ given by
\begin{equation}\label{on schedule g}
   g  = \tau_m \circ \cdots \circ \tau_2 \circ \tau_1 \in \cS_k.
\end{equation}
$$
   \tau_i =  \left\{   \begin{array}{ll}
             \text{the transposition } (a b)  & \mbox{if
               $\alpha_i=(a,b)$ and $\fs(\vs_i)=1$ } \\
                 \text{ the identity } & \mbox{if $\fs(\vs_i)=0   $}
             \end{array}
    \right. ,   
$$
compare with~\eqref{on schedule}.

Take the usual notation from  Hall-style matching theory,
and abbreviate the female ends as $\{1,2,\ldots,k\}$ and the male ends
as $\{1',2',\ldots,k' \}$.  Then $f_0$ induces the matching  from
$\{1,2,\ldots,k\}$ to $\{1',2',\ldots,k' \}$ with $a \mapsto a'$.
Now the $k$ paths under $f_0$ starting from
the  male ends
$\{1',2',\ldots,k' \}$ must eventually arrive at 
female ends $\{1,2,\ldots,k\}$.   Define the \emph{return matching}
$\gh$ by $\gh(a')=b$ if the path starting from the male end $a'$ first
arrives at the female end $b$.   This return matching $\gh$ is the same under
all
logics $\fs$ with $(\fs,\be) \in [(f_0,\be)]$.

Finally, for $(f,\be) \in H$, 
\begin{equation}\label{matching claim}
\pi_f \text{ relativized to } \{e_1,\ldots,e_k\} = \gh \circ g,   
\end{equation}
and of course, on each toggle class
$$
\dtv( \gh \circ g, \text{uniform}(\cS_k)) =   \dtv( g, \text{uniform}(\cS_k)).
$$
With hindsight, we observe that
the estimates of this section, and the previous Section 
 \ref{sect schedule}, 
have enabled us to dodge a very difficult consideration of
\emph{interlacement}
(of the $e_1,\ldots,e_k$ and $\vs_1,\ldots,\vs_1$);  see \cite{ABS}
for a study of interlacement.

\section{Sampling with $k$ Starts, to Prove Poisson-Dirichlet Convergence}
\label{sect sampling}

\subsection{Background, and Notation, for Flat Random Permutations}
\label{sect background}

An overall reference for the following material and history is  \cite{ABTbook}.
For a random permutation in $\cS_k$, with all $k!$ possible permutations equally
likely, for j=1,2,\dots, let
$$
  L_j  \equiv L_j(k) := \text{size of }j\text{-th longest cycle}
$$
with $L_j = 0$ if the permutation has fewer than $j$ cycles, so that always
$L_1(k)+L_2(k)+\cdots = k$. The notation $L_j  \equiv L_j(k)$  means that we consider the two notations equivalent, so that we can
use either, depending on whether or not we wish to emphasize the
parameter $k$. Write
\begin{equation}\label{def L}
   \bL \equiv \bL(k) := (L_1(k),L_2(k),\cdots), \ 
   \bbL \equiv \bbL(k) := \frac{\bL(k)}{k}, 
\end{equation}
so that $\barL_i \equiv \barL_i(k) := L_i/k$.
We use notation analogous to the above, systematically:  boldface gives a process, and overline specifies
normalizing,
so that the sum of the components is 1.

This 
paragraph, summarizing the convoluted 
history of the limit distribution for the length of the longest
cycle, begins with
Dickman's 1930 study of the largest prime factor of a random integer.
Dickman proved that for each fixed $u \ge 1$,
 $\Psi(x,x^{1/u})/x \to \rho(u)$, where $\Psi(x,y)$
counts
the $y$-smooth integers from 1 to $x$. The function  
$\rho$ is 
characterized by $\rho(u)=0$ for $u<0$, $\rho(u)=1$ for $0 \le u \le
1$,
and for all $u$, $u \rho(u)=\int_{u-1}^u \rho(t) \ dt$.  
In modern
language, writing $P^+ = P^+(x)$ for the largest prime factor of an
random integer chosen from 1 to $\lfloor x \rfloor$, 
Dickman's result is that
\begin{equation}\label{prime limit}
  \frac{\log P^+}{\log x} \to^d  X_1, \text{ where } \p(X_1 \le 1/u) = \rho(u) \text{
  for } u \ge 1.
\end{equation}
Later work by Goncharov (1944)  and Shepp and Lloyd (1966) showed
the corresponding result for random permutations, that
for every fixed $u \ge 1$, 
 $\p(L_1(k) < k/u) \to \rho(u)$. In modern language this is
\begin{equation}\label{largest limit}
L_1(k)/k \to^d X_1, \text{ where } \p(X_1 \le 1/u) = \rho(u) \text{
  for } u \ge 1.
\end{equation}
The random variable $X_1$ appearing in \eqref{prime limit}
and \eqref{largest limit} \emph{is the first coordinate} of the
Poisson-Dirichlet process;  the second coordinate corresponds to the
second largest prime factor, or second largest cycle length, and so
on.
For primes, the joint limit was proved by 
Billingsley (1972) \cite{billingsley72}, 
and for permutations, the joint limit was discussed by 
Vershik and Shmidt (1977) and Kingman (1977).
In these early studies, the Poisson-Dirichlet process appears as the
limit, but not in a form easily recognizable as either 
\eqref{PD density} or \eqref{def PD}.  A fun exercise for the reader
would be to prove that the distribution of $X_1$, as given by 
the cumulative distribution function in \eqref{prime limit}, 
together with the integral equation characterizing $\rho$,
is the same as the distribution of $X_1$ as given by its density,
which is the special case $k=1$ of \eqref{PD density}.
See \cite{budalect} for more on the Poisson-Dirichlet in relation to prime
factorizations, and \cite{ThreeCouplings} for more on the
Poisson-Dirichlet in relation to flat random permutations.

Returning to the process of  longest cycle lengths in 
\eqref{def L}, the joint distribution
 is most
easily understood by taking the cycles in ``age order".
Let
\begin{equation}\label{def A}
   A_j \equiv A_j(k) := \text{size of }j\text{-th eldest cycle}.
\end{equation}
Our notation convention has already told the reader that
$ \bA \equiv \bA(k) := (A_1(k),A_2(k),\cdots)$, and that $\bbA(k)=\bA(k)/k$.
    Here, the notion of age comes from canonical cycle notation: 1 is
written as the start of the first (eldest) cycle, whose length is
$A_1$, then the smallest $i$ not on this first cycle is the start of
the second cycle, whose length is $A_2$, and so on --- with $A_j := 0$
if the permutation has fewer than $j$ cycles.\footnote{In contrast 
with permutations on  $\{1,2,\ldots,N\}$,
similar to~\eqref{def A}, where age order comes from the
\emph{canonical} cycle notation, 
for shift-register permutations $\pi_f$, the oldest cycle is  \emph{not} the
cycle containing the lex-first $n$-tuple, $00\cdots0$. In fact, in a
random FSR, the cycle starting from $00\cdots0$ has exactly a
one-half chance to have length 1.  For permutations of a set lacking
exchangeability,  such as $\mathbb{F}_2^n$, the notion of \emph{age
order} requires \emph{auxiliary randomization}: the oldest cycle is
picked out by a \emph{random} $n$ tuple; conditional on this cycle,
with length $A_1<N$, choose an $n$ tuple uniformly at \emph{random}
from the remaining $(N-A_1)$ $n$-tuples not on the first cycle, to pick
out the second oldest cycle, whose length is $A_2$, and so on.}
It is easy to see that
$A_1$ is uniformly distributed in $\{1,2,\ldots,k\}$, and for each
$j=1,2,\ldots$, if there are at least $j$ cycles, then
$$  
      A_j(k)\text{ is uniformly distributed in }
\{1,2,\ldots,k-(A_1+\cdots+A_{j-1})\}.
$$
This very easily leads to a description of the limit 
proportions:  with $U,U_1,U_2,\ldots$ independent, uniformly distributed in (0,1),
\begin{equation}\label{to dist age ordered}
   \bbA := \frac{\bA(k)}{k} \to^d ( (1-U_1), U_1(1-U_2), U_1 U_2(1-U_3),\ldots).
\end{equation}
We write $\to^d$ to denote convergence in distribution, and we 
note that $U =^d 1-U$, where $=^d$ denotes equality in
distribution.
The  distribution of the process on the right side of 
\eqref{to dist age ordered} is named GEM, after
Griffiths~\cite{griffiths}, 
Engen~\cite{engen}, and McCloskey~\cite{mccloskey};
its construction is popularly referred to as ``stick breaking" although
stick breaking in general allows $U$ to take any distribution on
(0,1),
not just the uniform.

Convergence of processes, 
such as \eqref{to dist age ordered} and \eqref{to dist size ordered}, and our 
Theorem \ref{theorem 1} and Lemmas \ref{partition lemma} and \ref{permutation lemma},
are instances of convergence for stochastic processes with values in $\mathbb{R}^\infty$, with the usual compact-open topology, and as such, convergence of processes is equivalent to convergence to the finite-dimensional-distributions, of the first $r$ coordinates, for  each $r=1,2,\ldots$.

  Define
$$
\Delta = \{(x_1,x_2,\dots) \in [0,1]^\infty: x_1+x_2+\cdots =1\}.
$$
The (usual subspace) topology on $\Delta$ is the same as the metric topology from the $\ell_1$ distance, 
\begin{equation}
    d(  (x_1,x_2,\dots), (y_1,y_2,\dots) ) = \sum | x_i - y_i|,
\label{def ell 1}
\end{equation}
We write $\rank$ for the function on $\Delta$ which sorts, 
with largest first.  An example 
shows some of the subtlety of the preceeding considerations: let $\be_i \in \Delta$ be the $i^{th}$ standard basis vector --- all zeros apart from a 1 in the $i^{th}$ coordinate, and let $\bf{0}$ be the all zeros vector.  Note that ${\bf 0} \in [0,1]^\infty \setminus \Delta$, and in the larger space $[0,1]^\infty$, $\be_n \to \bf{0}$. But for $i \ne j, d(\be_i,\be_j)=1$, and the sequence $\be_1,\be_2,\ldots$ does not converge in $\Delta$.  The closure of $\Delta$ is the compact set
$\overline{\Delta} = \{(x_1,x_2,\dots) \in [0,1]^\infty: x_1+x_2+\cdots \le1\}$, and $\rank$ is also 
defined\footnote{$\rank$ is  \emph{not} defined on $[0,1]^\infty$  --- for example $\bx = (1/2,2/3,3/4,\ldots )$ does not have a largest coordinate.}
  on $\overline{\Delta}$; note that ${\bf 0} \in \overline{\Delta}$, and our $\be_n$ example shows that $\rank$ is \emph{not}
  continuous on $\overline{\Delta}$.
Donnelly and Joyce, \cite[Proposition 4]{Donnelly}, proved 
that $\rank$ 
is continuous on $\Delta$, 
observing that ``\dots in parts of the literature some of these results seem
already to have been assumed."   

By definition, a random $(X_1,X_2,\ldots) \in \Delta$ \emph{is} the
Poisson-Dirichlet process, or \emph{has} the Poisson-Dirichlet 
distribution\footnote{This PD is PD(1); mathematical geneticists work with a family of distributions, PD($\theta$), indexed by $\theta \in (0,\infty)$.},  PD, if for each $k=1,2,\ldots$, the
joint density of the first $k$ coordinates is given by
\begin{equation}\label{PD density} 
  f_k(x_1,x_2,\ldots,x_k) = \frac{1}{x_1x_2 \cdots x_k} \  \rho\left(
\frac{1-x_1-\cdots -x_k}{x_k} \right)
\end{equation}
on the region $x_1 > x_2 >\cdots > x_k > 0$ and $x_1+\cdots+x_k < 1$,
and zero elsewhere.
The Poisson-Dirchlet process may be constructed from the GEM process,
which appeared on the right side of 
\eqref{to dist age ordered}, by
sorting, with
\begin{equation}\label{def PD}
  (X_1,X_2,\ldots) =^d  \rank(( (1-U_1),U_1(1-U_2),U_1 U_2(1-U_3),\ldots) ).
\end{equation}   

For the process of largest cycle lengths in a random permutation, \eqref{def L},
the combination of the easy-to-see limit~\eqref{to dist age ordered}, and the continuity of \rank, and
the characterization~\eqref{def PD} of the Poisson-Dirichlet distribution, 
proves that as $k \to \infty$,
\begin{equation}
    \bbL(k) \to^d {\bf X} :=(X_1,X_2,\ldots),\text{ with PD distribution}.
\label{to dist size ordered}
\end{equation}

Our goal is to derive a new tool for proving the same PD convergence
as
in~\eqref{to dist size ordered}, but for non uniform permutations,
such as those arising from a random FSR.  It might benefit the reader
to jump ahead a little, and read the statement of Lemma 
\ref{permutation lemma}, 
and then 
the more technical
Lemma \ref{partition lemma},  
which has the meat of the argument used to prove  Lemma 
\ref{permutation lemma}.  We have stated Lemma \ref{partition lemma} in
a fairly general form, hoping that it may be useful in the context of other
combinatorial structures, and perhaps with limits other than the
Poisson-Dirichlet.

\subsection{The Partition Sampling  Lemma}

\begin{lemma}  \label{partition lemma}
First, suppose that for each $N$ along a sequence of $N$ tending to $\infty$
we have         a random set partition $\pi$ on 
$[N] :=\{1,2,\ldots,N\}$. Let $M_j \equiv M_j(N)$ 
be the size of the $j$-th
largest block of $\pi$, with $M_j :=0$ for $j$ greater than the number of blocks
of $\pi$,  so that    $M_1+M_2+\cdots = N$.  Let
$\bM(N)=(M_1(N),M_2(N),\ldots)$
and let $\bbM(N)=M(N)/N$.

Next, for each $k \ge 1$, take an ordered sample of size $k$, with replacement,
from $[N]$, with all $N^k$ possible outcomes equally likely.
Such a sample picks out an ordered 
$($by first appearance$)$
list of blocks of $\pi$, say
$\beta_1,\dots,\beta_r$, with $r \le k$. 
Let $C_j \equiv C_j(N,k)$ 
be the number of elements of 
the $k$-sample landing in the block $\beta_j$,  with $C_j := 0$ for 
$j>r$, so that $C_1+C_2+\cdots = k$. Let  $\bC \equiv \bC(N,k) =
(C_1,C_2,\ldots)$.

Finally, let $\bX = (X_1,X_2,\ldots)$ be \emph{any} random element of $\Delta$,
with $X_1 \ge X_2 \cdots \ge 0$, and let
$\bA(k) := (A_1(k),A_2(k),\cdots)$ be \emph{any} random elements
of $\mathbb{Z}_+^\infty$ for which $A_1(k)+A_2(k)+\cdots=k$, and such
that  $\bbA(k) := \bA(k)/k$ has
\begin{equation}\label{limit hypothesis}
  \text{ as } k \to \infty, \ \ \ \rank( \bbA(k)) \to^d \bX.
\end{equation}

Then, if for each fixed $k$,  as $N \to \infty$, we have
\begin{equation}  \label{k hypothesis} 
  \bC(N,k)  \to^d   \bA(k),
\end{equation}
it follows that  
$$
  \text{ as } N \to \infty, \ \ \  \bbM(N) \to^d \bX.
$$
\end{lemma}

\begin{proof} 
\def\whp{{\rm whp \ }}
Here is an outline of our proof.  We begin with an analysis of ``sampling using $k$ probes", leading to 
\eqref{uniform all i}, which gets coordinatewise nearness, with exceptional probability  O($1/k)$,  \emph{uniformly} over set partitions, which are indexed by $N$.
This is the crux of our proof;  the remainder is similar to Donnelly and Joyce, \cite[Proposition 4]{Donnelly}, on the continuity of \rank.   For an overview, writing whp to mean ``with high probability", and $\doteq$ to mean ``approximately equals, in $\ell_1$":
$$  \bX \  (\text{by } \eqref{limit hypothesis})  \doteq \whp \rank(\bbA) \ 
(\text{by } \eqref{k hypothesis})  = \whp \rank(\bbC) = \rank(\bbD)  \doteq \whp \rank(\bbM) = \bbM.
$$

Write the blocks of $\pi$ as $b_1,b_2,\ldots$, listed in nonincreasing order 
of size,  so that $M_i = |b_i|$.  
Write $p_i := M_i/N$,  so that $\bp :=
(p_1,p_2,\ldots) \equiv \bbM$ 
is a random probability distribution on the positive integers.
Let $D_j$ be the number of elements of the $k$-sample in $b_j$;
 the lists $C_1,C_2,\ldots$ and $D_1,D_2,\ldots$ 
represent
  the same multiset, apart from rearrangement, so that
\begin{equation}
\label{same multiset}   
\rank(  (C_1,C_2,\ldots) ) =  \rank(  (D_1,D_2,\ldots) ). 
\end{equation}
Write  $\bD \equiv \bD(N,k) := (D_1,D_2,\ldots)$, and $\bbD 
 \equiv \bbD(N,k) := \bD/k$, so that
$\bbD=(\barD_1,\barD_2,\ldots)$
and $\barD_i = D_i/k$.

Conditional
on the value of $\bp$, the joint distribution of $(D_1,D_2,\ldots)$
is exactly Multinomial$(k,\bp)$.  
We want to establish a form of \emph{uniformity} for the convergence
of $\bbD(k)$ to $\bp$.
The first step is to recall the usual proof that for Binomial sampling, 
with a sample of size $k$ and true parameter $p \in [0,1]$, 
the sample mean $\hat{p}$
converges to the true parameter $p$ --- because the proof provides a
quantitative bound.  Specifically,  Chebyshev's inequality gets
used, with
\begin{eqnarray}
\p( |  \hat{p}-p | \ge \delta )
      &=& \p( (\hat{p}-p)^2 \ge \delta^2) \nonumber \\
          & \le  & \frac{ \e (\hat{p}-p)^2}{ \delta^2} \nonumber  \\ 
          &  =   &  \frac{\var \hat{p}}{\delta^2}\nonumber  \\ 
          &  =   &  \frac{p(1-p)}{k \delta^2}\nonumber  \\ 
          &  \le &  \frac{p}{k \delta^2 }. 
\label{uniform binomial}         
\end{eqnarray}
In particular, conditional on any value for $\bp$, for $i=1,2,\ldots$,  with $p_i = \barM_i = M_i(N)/N$ in the role of $p$ for \eqref{uniform binomial},
$$
    \p(|\barD_i-\barM_i| \ge \delta \, | \, (p_1,p_2,\ldots ))
 \le  \frac{p_i}{k \delta^2 }.
$$
Hence, taking expectation to remove the
conditioning on $\bp$, and then using $\sum_i p_i=1$ to analyze the union bound, we have a good event $G$ (proximity in $\ell_\infty$) whose complement
\begin{equation}\label{uniform all i}
    G^c :=  
  ( \exists i,  |\barD_i-\barM_i| \ge \delta ) \ \
\text{ has } \p(G^c)
  \le  \frac{1}{k \delta^2 }.
\end{equation}

For $\bx \in \Delta$, $j \ge 1$ write $S_j(\bx)$ for the sum of the $j$ \emph{largest} coordinates of $\bx$.  Obviously
\begin{equation}\label{sampling 2}
\text{ for } \omega \in G, \ \ S_j(\bbM) \ge S_j(\bbD) - j \, \delta.
\end{equation}

Let $\varepsilon > 0$ be given, and fixed for the remainder of this proof.

Let  
\begin{equation}\label{def R j} 
R(j,\varepsilon) := \{\by = (y_1,y_2,\ldots) \in \Delta: \rank(\by) =
\bx =(x_1,x_2,\dots) \text{ has } x_1+\cdots+x_j >
1-\varepsilon\},
\end{equation}
the set of points in $\Delta$ where \emph{some} set of $j$ coordinates
sums to more than $1-\varepsilon$.  Note that $R(j,\varepsilon)$ is
invariant under permutations of the coordinates, including $\rank$.
Since $\Delta = \cup_j R(j,\varepsilon)$, 
and $\bX$ from~\eqref{limit hypothesis} is a random element of
$\Delta$,
there exists $j=j(\varepsilon) \ge 1$, depending on the distribution of $\bX$, such that 
\begin{equation}\label{X R j}
\p(\bX \in R(j,\varepsilon))
> 1- \varepsilon; 
\end{equation}
 fix such a value for $j$.   [When used in Lemma \ref{permutation lemma}, where the distribution of $\bX$ is Poisson-Dirichlet, \eqref{def PD} can be used to show that the minimal such $j$ is asympotically $\log (1/\varepsilon)$.]

Using the hypothesis~\eqref{limit hypothesis},
and observing that $R(j,\varepsilon)$ is an open set, (the \emph{open set} part of the Portmanteau Theorem on weak convergence implies that) we can pick and fix a finite
$k_0$    
such that for all $k \ge k_0$,
\begin{equation}\label{A(k) R j}
   \p(\bbA(k) \in R(j,\varepsilon)) >  1- \varepsilon.
\end{equation}
Using the hypothesis~\eqref{limit hypothesis} again,
 we can pick and fix a finite  $k_1 \ge k_0$ such that for each $k \ge k_1$,
there exists a coupling (see Dudley \cite{Dudley}, Real Analysis and Probability,
Corollary 11.6.4) such that the $\ell_1$ distance has
\begin{equation}\label{eps coupling}
  \p( d( \rank( \bbA(k)), \bX ) \ge \varepsilon) < \varepsilon.
\end{equation}

Next, intending to use  \eqref{uniform all i} 
with $\varepsilon/j$ used in the role of $\delta$, the upper bound is 
$1/(k \delta^2) = j^2/(k \varepsilon^2)$.   To have this upper bound be at most $\varepsilon$, and also be able to apply \eqref{eps coupling}, we take $k$ to be the maximum of $k_1$ and the ceiling of $ j^2/ \varepsilon^3$.

The value $k$ has been fixed, in the previous paragraph. Now, 
the convergence in hypothesis~\eqref{k hypothesis} involves the topologically discrete space
$\mathbb{Z}_+^k$, so the distributional convergence can be metrized by the total variation
distance, hence there exists a finite $N_0(k)$ such that 
for all $N \ge N_0(k)$, the total variation distance between distributions is at most
$\varepsilon$, and there exists a coupling with
$$ 
  \p( \bC(N,k) \ne \bA(k) ) \le \varepsilon.
$$ 
Of course  this same coupling and exceptional event yields  
$ \p(\rank( \bbC) \ne \rank(\bbA) ) 
 \le\varepsilon$, and  using also \eqref{A(k) R j},

$$ 
  \p(\rank( \bbC) = \rank(\bbA)  \text{ and } \bbC(N,k) \in R(j,\varepsilon)) > 1-2 \varepsilon.
$$ 
But then \eqref{same multiset}, and the permutation invariance of $R(j,\varepsilon)$ converts the above into
\begin{equation}\label{D near rank A}
  \p( \rank(\bbD) = \rank(\bbA)   \text{ and  } \bbD(N,k) \in R(j,\varepsilon)) > 1-2 \varepsilon.
\end{equation}

Next, observe that $\bbD(N,k) \in R(j,\varepsilon)$ and $G$ from \eqref{uniform all i} with $\delta = \varepsilon/j$
imply that, each of the $j$ indices $i$ for $\bbD(N,k) \in R(j,\varepsilon)$ has $|M_i -  D_i| < \delta$, so the sum of those $j$ coordinates of $\bbM$ is at least $S_j(\bbD) - j \, \delta = S_j(\bbD) - \varepsilon > 1 - 2 \varepsilon$  (as observed in \eqref{sampling 2})),
and the sum of the other (outside the chosen $j$) coordinates of $\bbM$ is at most $2 \varepsilon$, while the sum of the other (outside the chosen $j$) coordinates of $\bbD$ is at most $\varepsilon$.  Hence, the $\ell_1$ distance is at most
$4 \varepsilon$, accounted for by $j \delta = \varepsilon$, from the $|M_i - D_i|$ with $i$ among the chosen $j$, plus $2 \varepsilon + \varepsilon$ using $ |M_i - D_i| \le M_i + D_i$ on the other coordinates, outside the chosen $j$.  This result was that $d(\bbM,\bbD) < 4 \varepsilon$.  Now $\bbM = \rank(\bbM)$ by construction, but due to sampling noise, maybe $\bbD \ne \rank(\bbD)$.  However,  since $\rank$ is a contraction, we have $d(\bbM,\rank(\bbD)) < 4 \varepsilon$.

Putting it all together, for any $N \ge N_0$, the union of the  exceptional events from \eqref{uniform all i} ( $\bbM$ near $\bbD$, coordinatewise, with $\p(G^c) \le \varepsilon$), from \eqref{eps coupling} ($\rank(\bbA)$ near $\bX$), and from \eqref{D near rank A}
($\bbD$ equals $\rank(\bbA)$, in $R(j,\varepsilon)$) has probability at most $4 \varepsilon$, and outside this exceptional event, $\bbM$ is at most
$4 \varepsilon$ away from $\rank(\bbD) = \rank(\bbA)$, which in turn is at most $\varepsilon$ away from $\bX$.
In summary, there are couplings so that
$$ \forall N \ge N_0,  \ \p( d(\bbM,\bX) > 5 \varepsilon) < 4 \varepsilon.   $$

\end{proof}

\subsection{The Permutation Version of the Sampling Lemma}
\label{sect perm version}

\begin{lemma}
\label{permutation lemma}
Suppose that for a sequence of $N$ tending to $\infty$ we have a random permutation $\pi$ on 
$[N] :=\{1,2,\ldots,N\}$. Let $M_j  \equiv M_j(N)$ 
be the size of the $j$-th
largest cycle of $\pi$, with $M_j :=0$ for $j$ greater than the number of cycles
of $\pi$,  so that $M_1+M_2+\cdots = N$.

Given $k \ge 1$, take an ordered sample of size $k$, with replacement,
from $[N]$, that is, 
$e_1,\ldots,e_k$ 
with all $N^k$ possible outcomes equally likely.
Let $\sigma$ be $\pi$ relativized to $e_1,\ldots,e_k$, as defined at
the start of Section \ref{sect tangle}.

Now suppose that, for each fixed $k \ge 1$,
\begin{equation} \label{each k permutation}  
   \forall \tau \in \cS_k, \ \text{ as } N \to \infty, \ \ 
    \p(\sigma = \tau)   \to  1/k! .
\end{equation}

Then, as $N \to \infty$,
\begin{equation}\label{thm 2 conclusion}
    (M_1(N)/N,M_2(N)/N,\ldots) \to^d \bX =(X_1,X_2,\ldots), 
\end{equation} 
where $\bX$ has the Poisson-Dirichlet distribution, as in~\eqref{def
  PD} and~\eqref{to dist size ordered}.
\end{lemma}
\begin{proof}

Take  the processes $\bA(k)$ of cycle lengths, in age order, as given by 
\eqref{def A}, for uniform random permutations in $\cS_k$, to  serve 
as the
random elements in the hypotheses~\eqref{limit hypothesis} and
\eqref{k hypothesis} of  
Lemma \ref{partition lemma}.  This requires using the
Poisson-Dirichlet distribution, for $\bX$ in~\eqref{limit hypothesis}.

Fix $k$.
Then~\eqref{each k permutation} holding for each $\tau \in \cS_k$ implies that
the distribution of $\sigma$ is close, in total variation distance, to the uniform
distribution on $\cS_k$.   On the event, of probability
$\frac{N-1}{N} \cdots \frac{N-(k-1)}{N} \to 1$,  that the $k$-sample with
replacement from the
$N$ population has $k$ distinct elements, the counts $C(N,k)$ from
Lemma \ref{partition lemma} agree exactly with the cycle lengths in $\sigma$.
Hence hypothesis~\eqref{each k permutation} implies the hypothesis
~\eqref{k hypothesis}.

\end{proof}

\section{Putting it All Together: The Proof of Theorem~\ref{theorem 1}}
\label{sect together}

We now have established all the ingredients needed  for our proof 
of  Theorem~\ref{theorem 1}.   First, the conclusion~\eqref{thm 1
  conclusion} 
of Theorem~\ref{theorem 1} is
exactly the conclusion~\eqref{thm 2 conclusion} from Lemma
\ref{permutation lemma}.\footnote{There is a small shift of notation; 
in Section~\ref{sect sampling} we had to
  deal with \emph{both} FSR permutations and flat random permutations. 
So in  Section~\ref{sect sampling}, 
instead of  $\bL$ for the process of largest FSR cycles lengths,
 $\bM$ names the process of
  largest cycle lengths for an FSR permutation, and $\bL$ names the
  corresponding process for flat random permutations.}  To prove
Theorem~\ref{theorem 1}, it only remains to establish that the random
FSR model~\eqref{def pi f} satisfies the 
hypothesis~\eqref{each k permutation} of 
Lemma
\ref{permutation lemma}.
 
Fix $k$ for use in~\eqref{each k permutation}.
The uniform choice of $(f,\be) \in S_{n,k}$ determines
$\pi_f$ and the random sample $e_1,\ldots,e_k$  --- for convenience in
Lemma \ref{permutation lemma} we labeled the set $\mathbb{F}_2^n$ with
the
integers 1,2,$\ldots,N$.  
Let an arbitrary $\varepsilon >0$ be given.
Fix $m = m(k,\varepsilon)$ as per Lemma \ref{schedule lemma}, so that
with high probability, a random schedule of length $m$ over the
alphabet
of size ${k \choose 2}$ is $\varepsilon$-good.

We will take $t=N^{.6}$, recalling that $N=2^n$.
By Theorem~\ref{kt theorem}, 
for sufficiently large $n$, 
 on a good event $\Gkt$ of probability at least  $1-\varepsilon$,
the two-dimensional 
process $\bXv$
of indicators of vertex 
 repeats, in
$\seg(f,e_1,k(t+n))$,
agrees with the two-dimensional process $\bX$
of indicators of leftmost $(n-1)$-tuple repeats for coin tossing; 
and cutting, to produce $\be$ and $k$ segments, causes no unwanted
side effects.    
Then, by the Chen-Stein  method as given by Theorem 3 of
\cite{AGG89}  (with a survey of applications to sequence repeats given by Section 5 of
\cite{AGG90}, and details for the sequence repeats problem
given in (39)--(40) of \cite{AMRW}), for sufficiently large $n$  the 
total
variation distance between $\bX$ and $\bXp$ is at most $\varepsilon$,
where $\bXp$ has the same marginals as $\bX$, but all coordinates
mutually
independent.
Combined, the total variation distance between $\bXv$ and $\bXp$ is
arbitrarily small, at most $2 \varepsilon$.

The indicator of the happy event $H$ is a functional of the process 
$\bXv$,  so we can approximate $\p(H)$, with an additive error of at
most $2 \varepsilon$, by evaluating the same functional, applied to $\bXp$.
The required estimates for this independent process are routine, via
computations of the expected number of arrivals in various regions as in 
Section \ref{sect toggling},\footnote{These arguments take two forms:  1) if the expected
  number of arrivals is small, specifically, less than $\delta$, 
then the probability of (no arrivals)
  is large, specifically, greater than $1-\delta$, and 2) if the
  expected number of arrivals is sufficiently large, specifically,
  some $\lambda >1$, \emph{and the indicators of arrivals are
  mutually independent}, then the probability of (no arrivals) is
  small,
specifically, at most $e^{-\lambda}$.  It is precisely the role of the
Chen-Stein method to provide the required independence.}
and we have already provided most of the details,  in discussing
\eqref{need 1} and~\eqref{need 3}.  Additionally, one must check 
that  the schedule resulting from
use of~\eqref{def choice} is close, in total variation distance, to
the flat random choice  in the hypothesis of 
Lemma \ref{schedule lemma}; we omit the relatively easy details. 

To summarize, 
we picked $k$ for use in  Lemma \ref{permutation lemma}, 
then fixed an arbitrary $\varepsilon>0$,
then picked $m$ via  
Lemma \ref{schedule lemma}.\footnote{\label{summary footnote} In a sense,  Lemma
  \ref{permutation lemma}
encapsulates a relation between an arbitrary $\varepsilon>0$, and $k$,
hiding the full  
program:  given $\varepsilon>0$ to govern \emph{being close with high
  probability},  pick a single $k$ large
enough that the $k$-sampled-and-relativized
permutation being  close to uniform in $\cS_k$ would
imply 
that the large cycle process for FSR permutation is 
close to the PD, then pick a single $m$ to work for this $k$ and $\varepsilon$, 
then finally pick $n_0$, the notion of \emph{sufficiently large} $n$,
to work for this $k,m$ and $\varepsilon$.  The briefest 
summary is:
given
$\varepsilon$, pick $k$, then $m$, then $n_0$.
}
    For large $n$, the process of vertex
repeats
among the $k$ segments of length $t$ is controlled,
 via  comparison of $\bXv,\bX,\bXp$,  
showing that most 
$(f,\be)$ lie in $H$,  and furthermore, the event $H^* \equiv
H^*(\varepsilon)  \subset H$,
that the chosen potential toggle vertices $\vs_1,\ldots,\vs_m$
pick out a $\varepsilon$-good schedule, has $\p(H^*)>1-4
\varepsilon$. (Attributing $2 \varepsilon$ to $\dtv(\bXv,\bXp)$, $
 \varepsilon$ to $\p(H^c)$, and $
 \varepsilon$ to $\p(H \setminus H^*)$    .)
 Section~\ref{sect schedule}  shows that, on $H^*$, the
 settings of $f$ at its toggle vertices 
induce a nearly flat random matching between segment starts and ends,
and~\eqref{matching claim} in Section~\ref{sect tangle} lifts this to 
show that $\pi_f$
relativized to $e_1,\ldots,e_k$ is a nearly flat random
permutation in $\cS_k$.  Thus the combination of 
Section~\ref{sect schedule} and~\ref{sect tangle} shows that, on $H^*$, 
 on each equivalence class 
$[(f,\be)] \in H^*$, the total variation distance to the uniform
distribution
on $\cS_k$ is at most $\varepsilon$.  Hence, averaging over the
classes
in $H^*$, and allowing distance 1 for the at most $4 \varepsilon$ of
probability mass outside of $H^*$, we get that
for our fixed $k$, for arbitrary $\varepsilon$, for all sufficiently
large $n$,  $\dtv(\sigma, $ uniform$(\cS_k)) = \frac{1}{2}
\sum_{\tau \in \cS_k} | \p(\sigma = \tau) - \frac{1}{k!}| < 5 \varepsilon$,
which establishes~\eqref{each k permutation}.
This completes the proof.

\section{Acknowledgments}
We would like to acknowledge numerous helpful
discussions with Danny Goldstein, Max Hankins and Jay-C
Reyes.

\clearpage

\nocite{ABS}
\nocite{punctured}
\nocite{ThreeCouplings}
\nocite{Diaconis}
\nocite{Maurer}
\bibliography{article}
\bibliographystyle{plain}

\end{document}